%% file: perm.tex
\documentclass[a4paper,11pt]{amsart}

\usepackage{fullpage}
\usepackage{amsfonts,amsmath}
\usepackage{amssymb} 
\usepackage{amsthm}
\usepackage{ifthen}
\usepackage{paralist}
\usepackage{xcolor}
\usepackage{xparse}
\usepackage{tabularx}
\usepackage{diagbox}
\usepackage{calc}
\usepackage{xfrac}

\usepackage{caption}
\usepackage{subcaption}
\usepackage{tikz}
\usetikzlibrary{arrows}
\usetikzlibrary{backgrounds} 

\makeatletter
\g@addto@macro\bfseries{\boldmath}
\makeatother

\makeatletter
\g@addto@macro\@floatboxreset{\centering}
\makeatother

\ExplSyntaxOn
\NewDocumentCommand{\myProcessList}{mm}{
  \clist_map_inline:nn{#1}{#2{##1}}
}
\ExplSyntaxOff

\newcounter{parcounter}
\makeatletter
\def\paragraphtitle{\@ifstar{\@paragraphtitle}{\@@paragraphtitle}}
\def\@paragraphtitle#1{%
  \medskip\medskip%
  \stepcounter{parcounter}%
  \boldmath\textbf{\S\arabic{parcounter} #1}%
  \unboldmath%
}
\def\@@paragraphtitle#1{
  \medskip\medskip%
  \stepcounter{parcounter}%
  \phantomsection\addcontentsline{toc}{subsection}{\S\arabic{parcounter} #1}%
  \boldmath\textbf{\S\arabic{parcounter} #1.}%
  \unboldmath%
}
\newcommand{\parlabel}[1]{%
  \def\savecurrentlabel{\@currentlabel}%
  \def\@currentlabel{\arabic{parcounter}}%
  \label{#1}%
  \def\@currentlabel{\savecurrentlabel}%
}
\makeatother
\newcommand{\parref}[1]{\S\ref{#1}}
\newcommand{\doparref}[1]{%
 \stepcounter{curarg}%
 \ref{#1}%
 \ifthenelse{\value{curarg} < \value{numarg}}{,}{}%
}

\theoremstyle{definition}
\newtheorem{definition}{Definition}[parcounter]
\newtheorem{theorem}[definition]{Theorem}
\newtheorem*{theorem*}{Theorem}
\newtheorem{proposition}[definition]{Proposition}
\newtheorem{lemma}[definition]{Lemma}
\newtheorem{corollary}[definition]{Corollary}

\newcommand{\N}{\mathbb{N}}
\newcommand{\Z}{\mathbb{Z}}

\newcommand{\G}{\mathcal{G}}
\newcommand{\I}{\mathcal{I}}
\renewcommand{\P}{\ifmmode\mathcal{P}\else\textparagraph\fi}

\renewcommand{\S}{\ifmmode\mathcal{S}\else\textsection\fi}

\newcommand{\overbar}[1]{\overline{#1}}

\def\innerseqx(#1;#2,#3){\{#1\}_{#2=#3}^\infty}

\newcommand{\fr}[1]{\mathfrak{#1}}
\renewcommand{\bf}[1]{\mathbf{#1}}
\newcommand{\xqedhere}[2]{%
  \rlap{\hbox to#1{\hfill\llap{\ensuremath{#2}}}}%
}
\newcommand{\gen}[1]{\langle#1\rangle}
\newcommand{\code}[1]{\langle#1\rangle}
\newcommand{\Mod}[1]{\ \mathrm{mod}\ #1}

\newcounter{numarg}
\newcounter{curarg}
\makeatletter
\newcommand{\countargs}{\stepcounter{numarg}\@gobble}
\makeatother
\newcommand{\docycle}[1]{%
 \stepcounter{curarg}%
 #1%
 \ifthenelse{\value{curarg} < \value{numarg}}{%
  \;\,%
 }{}
}

\DeclareDocumentCommand{\cycle}{m}{%
 \setcounter{numarg}{0}%
 \myProcessList{#1}{\countargs}%
 \setcounter{curarg}{0}%
 (\myProcessList{#1}{\docycle})%
}

\DeclareMathOperator{\Sym}{Sym}
\DeclareMathOperator{\id}{id}
\DeclareMathOperator{\rng}{rng}
\DeclareMathOperator{\succf}{succ}

\DeclareMathOperator{\Part}{Part}
\DeclareMathOperator{\Perm}{Perm}
\DeclareMathOperator{\Permcf}{Perm_{\textsc{CF}}}
\DeclareMathOperator{\Aut}{Aut}
\DeclareMathOperator{\dom}{dom}
\DeclareMathOperator{\Eqv}{Eqv}

\begin{document}

\title{On permutations with decidable cycles}
\author{Tobias Boege}
\address{Otto-von-Guericke-Universit\"at Magdeburg\\ Magdeburg, Germany}
\email{tboege@st.ovgu.de}
\date{\today}

\begin{abstract} Recursive permutations whose cycles are the classes of a
decidable equivalence relation are studied; the set of these permutations is
called $\Perm$, the group of all recursive permutations $\G$. Multiple
equivalent computable representations of decidable equivalence relations are
provided. $\G$-conjugacy in $\Perm$ is characterised by computable isomorphy
of cycle equivalence relations. This result parallels the equivalence of
cycle type equality and conjugacy in the full symmetric group of the natural
numbers.

Conditions are presented for a permutation $f \in \G$ to be in $\Perm$ and
for a decidable equivalence relation to appear as the cycle relation of a
member of $\G$. In particular, two normal forms for the cycle structure of
permutations are defined and it is shown that conjugacy to a permutation in
the first normal form is equivalent to membership in $\Perm$. $\Perm$ is
further characterised as the set of maximal permutations in a family of
preordered subsets of automorphism groups of decidable equivalences.

Conjugacy to a permutation in the second normal form corresponds to
decidable cycles plus decidable cycle finiteness problem. Cycle decidability
and cycle finiteness are both shown to have the maximal one-one degree of
the Halting Problem. Cycle finiteness is used to prove that conjugacy in
$\Perm$ cannot be decided and that it is impossible to compute cycle
deciders for products of members of $\Perm$ and finitary permutations. It
is also shown that $\Perm$ is not recursively enumerable and that it is not
a group.
\end{abstract}
\maketitle

\section{Introduction}
An equivalence relation over the natural numbers is \emph{decidable} if
there is a Turing machine which decides for every pair $(x, x')$ of numbers
whether they are related or not. The set of all recursive permutations of
the natural numbers is denoted $\G$ in this paper. We consider the subset
$\Perm$ in $\G$ of recursive permutations whose orbits, or \emph{cycles},
are the classes of a decidable equivalence relation. The present paper
studies algorithmic and algebraic questions about $\Perm$ which arise
naturally from this definition, which relates equivalence relations and
permutations.

Identify a permutation $f$ of $\N$ with the digraph on vertices $\N$ whose
arrows are given by the mapping $x \to f(x)$. A permutation is in $\G$ if
its digraph is locally explorable by a Turing-computable algorithm. For
the permutations in $\Perm$ a more global class of questions can be decided
in addition, namely for any pair of numbers whether they belong to the same
weakly connected component of the digraph. The weakly connected components
in the digraph view correspond to the cycles of the permutation and we use
these two terms synonymously.

The set $\Perm$ appears in an attempt to transfer a well-known theorem
from Group Theory to Recursion Theory. This theorem states that in a
symmetric group, such as $\Sym \N$, conjugacy is equivalent to cycle
type equality. Cycle type equality of two permutations is the condition
that for each countable cardinal, the both permutations have the same
number of cycles of size that cardinal. An attempt at this transfer has
been made by Kent in 1962:

\begin{theorem*}[{\cite[Thm.~1.7]{kent62}}]
\phantomsection\label{theorem:kent} In the group $\G$, a cycle type
class is also a conjugacy class iff it is the cycle type class of a
permutation with finitely many infinite cycles.
\end{theorem*}

Define the set $\G_1$ to consist precisely of those recursive permutations
with finitely many infinite cycles. Then Kent's theorem states that in
$\G_1$ two permutations are $\G$-conjugate iff they have the same cycle
type and there is no proper superset $\G_1 \subsetneq \G' \subseteq \G$
which is closed under cycle types and where this theorem holds, too. This
formulation takes the same form as the one for full symmetric groups, but
one might argue that $\G_1$ is a too small subset of $\G$. We will show, for
instance, that cycle decidability and cycle finiteness are trivial problems
in $\G_1$ (Propositions~\ref{prop:fininfdec}, \ref{prop:fininfcf}). Thus one
asks why the notions of conjugacy and cycle type equality drift apart in
$\G$. The crucial observation is that by changing the group from $\Sym \N$
to $\G$, the notion of conjugacy changes. Conjugation in $\G$ is always
afforded by a \emph{recursive} permutation --- it gets a constructive
character. Cycle type equality, on the other hand, remains non-constructive
in Kent's theorem. If we define an effective version of cycle type equality,
we are not restricted by the second part of Kent's theorem anymore and may
find a bigger set in which the new formulation of the theorem holds. To
make cycle type equality effective, one first needs a witness for the
condition of cycle type equality. Such a witness would be a bijection
between the weakly connected components which preserves the size of each
component, or, alternatively, a bijection between the sets of vertices which
respects weakly connected components in both directions. If we denote the
equivalence relation whose equivalence classes are the weakly connected
components of $f$ by $\equiv_f$, then we require a permutation $\theta$ of
$\N$ such that $x \equiv_f x' \Leftrightarrow \theta(x) \equiv_g
\theta(x')$. This takes the form of an equivalence isomorphism; see
Figure~\ref{fig:conjcycletypeeq} for an illustration of witnesses for
conjugacy and cycle type equality. The notions of equivalence relation and
isomorphism thereof can easily be transferred into Recursion Theory and
yield the desired definition of effective cycle type equality. The present
paper starts by giving a number of possible definitions for decidable
equivalences which parallel characterisations of equivalence relations in
non-constructive Mathematics, and shows that they are equivalent, too, in
Recursion Theory. Based on this solid notion of decidable equivalence, we
are interested in recursive permutations $f$ whose relation $\equiv_f$ is
decidable, as only for those permutations effective cycle type equality can
be defined in the language of Recursion Theory. The set of these
permutations is exactly $\Perm$, and indeed one obtains an analogue to the
theorem in $\Sym \N$, our Theorem~\ref{theorem:permconj2}, which states that
in $\Perm$ $\G$-conjugacy and effective cycle type equality are equivalent.
This theorem has two advantages over the one in $\G_1$:
\begin{inparaenum}[(1)] \item the equivalence is constructive, i.e. a
witness for either condition can be converted into one for the other by
Turing-computable algorithms, and \item $\G_1$ is properly (and trivially)
contained in $\Perm$, by Proposition~\ref{prop:fininfdec}.\end{inparaenum}

\begin{figure}
\tikzset{
  ball/.style={
    circle,
    minimum size=0.7em,
    inner sep=0,
    font=\footnotesize
  }
}
\renewcommand{\thesubfigure}{{\normalfont\alph{subfigure}}}
\makebox[\linewidth][c]{%
\subcaptionbox{Conjugacy \label{fig:conj}}{%
  \input{witnesses_a.tex}%
}
\hskip 1.0cm
\subcaptionbox{Cycle type equality \label{fig:cycletypeeq}}{%
  \input{witnesses_b.tex}%
}
}
\caption{Witnesses of conjugacy and cycle type equality. Conjugacy of $f$
and $g$ is witnessed by an isomorphism $h$ between their respective
digraphs, cycle type equality by an isomorphism $\theta$ between the
equivalence relations of \emph{undirected} reachability in their respective
digraphs. The solid circles in \ref{fig:conj} are vertices, the ellipses in
\ref{fig:cycletypeeq} are weakly connected components.}
\label{fig:conjcycletypeeq}
\end{figure}
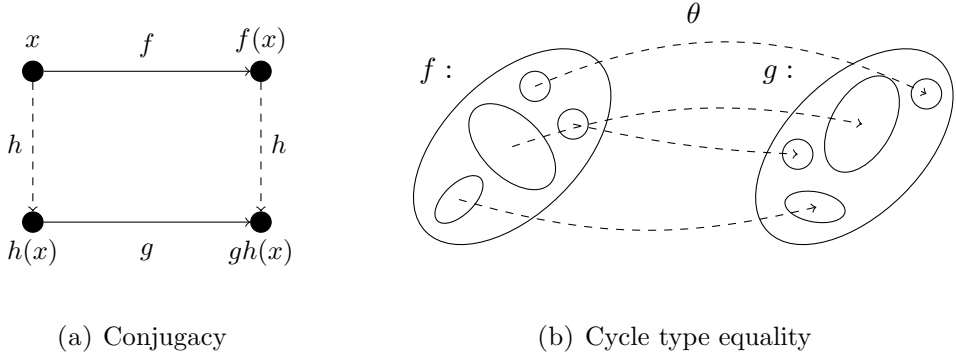

After the theorem about the equivalence of $\G$-conjugacy and effective
cycle type equality in $\Perm$ is established, section~\ref{sec:charperm}
gives a number of characterisations of $\Perm$. The two main ideas come
from the two ways to approach $\Perm$, as per its definition. The
first is the question whether a given $f \in \G$ has decidable cycles,
the second is whether a given decidable equivalence relation is realisable
as the cycle equivalence of a permutation in $\G$. In this process,
\parref{par:cycledec} collects evident sufficient and equivalent extrinsic
criteria for cycle decidability. The approach of \parref{par:normal} is to
characterise cycle decidability intrinsically through normal forms for the
cycle structure of permutations. Two forms, the \emph{normal} and the
\emph{semi-normal} form, are defined. A corollary to
Theorem~\ref{theorem:decnormal} is that a permutation has decidable
cycles iff it is $\G$-conjugate to a permutation in normal form. Indeed if
a permutation is in normal form, a decider for its cycles can be extracted
from the permutation alone. Normality, in this light, is a uniformity
condition. Then, \parref{par:permutability} gives sufficient and
equivalent criteria for a decidable equivalence to be \emph{permutable},
i.e.~to appear as the cycle equivalence of a recursive permutation. It is
shown that decidable and permutable equivalences are in bijection with
recursive normal-form permutations. The techniques developed in this
subsection give powerful tools to construct permutations from mere
equivalence relations which only encode the orbits of a permutation, not
the cyclic structure. Especially Corollary~\ref{coro:pairrhochar} proves
to be useful in existence theorems such as Proposition~\ref{prop:nonpermut},
Lemma~\ref{lemma:infcycle} and Theorem~\ref{theorem:interred}. One further
characterisation of $\Perm$ is of order-theoretic nature. It is shown that
the elements of $\Perm$ are exactly the permutations in a preordered subset
of the automorphism group of some decidable equivalence which are maximal
with respect to cycle inclusion. Here, the characterisation via maximality
does not use notions of computability; these occur only in the setting, i.e.
the choice of the preordered sets.

While the normal form deals with cycle decidability, the semi-normal form
specifies an incompatibility between the structure of finite and infinite
cycles which makes it possible, in addition to deciding the cycles, to
decide the \emph{cycle finiteness problem} of the semi-normal
permutation, that is for each number $x$ to determine if its cycle is finite
or infinite. It has previously been shown by Lehtonen~\cite{lehton09} that
there are recursive permutations which can be defined by fairly elementary
formulae but whose cycle finiteness problem is undecidable. In
\parref{par:cyclefinite} in section \ref{sec:cfunsolv}, it is proved that
the set of permutations for which this problem is decidable is precisely
the union of $\G$-conjugacy classes of recursive semi-normal permutations.
Remarkably, the sets of permutations for which the cycle decidability and
cycle finiteness problems are solvable, each possess a set of generators
with respect to $\G$-conjugacy, namely the normal and the semi-normal
permutations, such that the problems are uniformly solvable on these
generators.

The fact that cycle finiteness in general is undecidable in $\Perm$ can be
applied in reductions. For example, deciding cycle finiteness is an instance
of deciding conjugacy in $\Perm$, which implies that the latter is also
undecidable. As a corollary to a result by van Leeuwen~\cite{leeuwen15},
we obtain that $\Perm$ is not recursively enumerable. In
\parref{par:diffcycledec} it is shown that cycle decidability problems in
$\G$ are at most as hard as the Halting Problem for Turing machines and
that this upper bound is attained. It is also proved that the classes of
cycle decidability problems and cycle finiteness problems in $\G$ are
reducible to each other. The last subsection concerns multiplicative closure
of $\Perm$. It is shown that a product of a member of $\Perm$ and a finitary
permutation is again in $\Perm$. In contrast to
Theorem~\ref{theorem:permconj2}, this fact has no algorithmic content. The
third part of Corollary~\ref{coro:permclosure} states that there is no
computable mapping which associates a decider for the cycles of $af \in
\Perm$ to every triple $(a, f, \pi)$, where $a \in F$, $F$ the set of
finitary permutation, $f \in \Perm$ and $\pi$ a decider for the cycles of
$f$. Such an algorithm exists, however, for semi-normal permutations. By
Corollary~\ref{coro:permnotgroup}, a product of two arbitrary (non-finitary)
members of $\Perm$ does not necessarily reside in $\Perm$ which shows that
$\Perm$ is not a group. This may be seen as an ultimate defect in trying to
recover the aforementioned theorem on conjugacy in symmetric groups inside
Recursion Theory. Theorem~\ref{theorem:permconj2} is an entirely
constructive version of this theorem, but its domain, $\Perm$, is missing
the group structure from the non-constructive original, which means that the
notion of conjugacy has to be borrowed from the ambient group $\G$. Kent's
set $\G_1$ suffered from the same problem. Indeed this failure is
predetermined by a result on the composition series of $\G$, also due to
Kent: no set strictly between $F$ and $\G$ which is closed under
$\G$-conjugation, i.e.~normal, can be a group.

If not explicitly introduced, the notation follows \cite{rogers87}, which
can also serve as the primary resource for recursion-theoretic facts used
without citation. The set of natural numbers, including $0$, is denoted $\N$
and the set of positive integers as $\N^+$. We denote the domain of a partial
function $f$ by $\dom f$ and its range by $\rng f$. If a composition of
functions is applied to an object, we save parentheses: instead of $f(g(x))$
or $(fg)(x)$ only the ``application to $x$'' parentheses are written: $fg(x)$.
By $\code{x,y}$ we denote a recursive \emph{pairing function} which maps pairs
$(x,y)$ to numbers bijectively. Instead of $f(\code{x, y})$ we write
$f\code{x, y}$. Existential and universal quantifiers are taken over the
natural numbers if no set is given. The complement of a set $A$ is also
understood as the complement in the natural numbers and written as
$\overbar{A}$. Usually we do not discriminate between a partial recursive
function $f$ and a program, or \emph{G\"odel number}, $\bf f$ under the
standard numbering of partial recursive functions: $\varphi_{\bf f} = f$.
The associated standard numbering of recursively enumerable sets is
$W_x := \dom \varphi_x$. When the function $f$ is used in the description of
an algorithm, it is implied that any G\"odel number for $f$ would suffice.
Most proofs in this paper are constructive and yield, by virtue of the
Church-Turing Thesis, a (Turing-computable) algorithm. This is usually
indicated by an \emph{effectiveness} addition in the formulation of the
statement. The term ``uniformly effectively'' is used throughout the text
according to \cite[\textsection~5.5]{rogers87}. The term ``cycle'' is used
with two different meanings in this paper. Firstly, it can mean an orbit
of the group action a permutation affords on its domain by function
application. Secondly it abbreviates \emph{cyclic permutation}, a
permutation whose support consists of at most one orbit. Commonly $f, g, h$
denote recursive permutations, often members of $\Perm$, $\sigma, \rho,
\psi$ partial recursive functions encoding an object of importance in a
limited scope, $\pi, \gamma$ deciders for equivalences $\Pi, \Gamma$, and
$x, y, z$ natural numbers.

\subsection*{Acknowledgement.} The author wishes to thank Thomas Kahle for
his support of this paper and discussion of earlier versions which helped
to improve the presentation.

\section{Decidable equivalences and their isomorphisms}

\paragraphtitle{Decidable equivalence relations} \parlabel{par:decequiv}
This subsection gives a number of possible constructive representations for
decidable equivalence relations and shows that they all are effectively
equivalent: for each pair of representations there is an algorithm which
transforms one into the other.

A set $\Pi \subseteq \P(\N)$ of subsets of $\N$ is called a \emph{partition}
if the sets in $\Pi$ are non-empty, pairwise disjoint and their union is
$\N$. The elements of $\Pi$ are called its \emph{blocks}. Denote by
$P(x, \Pi)$ the unique block $P \in \Pi$ with $x \in P$. Usually we
abbreviate this to $P(x)$ if there is no danger of confusion. It is well
known that the concepts of partition and equivalence relation are
equivalent: to each partition $\Pi$ there is the equivalence $x \equiv_\Pi
x' :\Leftrightarrow P(x) = P(x')$ and the classes of every equivalence form
a partition. We write $x \Pi x'$ instead of $x \equiv_\Pi x'$ and speak
of $\Pi$ also as an equivalence relation.

\begin{definition} An equivalence relation $\Pi$ over $\N$ is
\emph{decidable}, if for every $x, x' \in \N$, the predicate $x \Pi x'$ is
decidable.
\end{definition}

Let $\fr t$ and $\fr f$ be two distinct natural numbers, representing
\emph{true} and \emph{false}. We adopt the following syntax for predicates
$p$:
\[
  [p(x_1, \dots, x_n)] := \begin{cases}
    \fr t, & \text{$p(x_1, \dots, x_n)$ is true}, \\
    \fr f, & \text{else},
   \end{cases}
\]
Then we can state the decidability of $\Pi$ as follows: there exists a
recursive function $\pi$ of two variables such that $\pi(x, x') =
[x \Pi x']$.

For Recursion Theory, the need for representation of objects as natural
numbers arises. A partition $\Pi$ is identified with its set of blocks. By
the symbol $\equiv_\Pi$ we denote the same equivalence relation as $\Pi$ but
represented in the customary fashion as a set of ordered pairs $(x, x')
\in \; \equiv_\Pi \; :\Leftrightarrow x \Pi x'$. Then we see that $\Pi$ is
a decidable equivalence iff $\equiv_\Pi$ is a recursive set (of pairs) in
the usual sense of \cite[\textsection~5.3]{rogers87}. This subsection shows
that the distinction of these two representations is immaterial to
computability. The following lemma is obvious from the $s^m_n$ Theorem:

\begin{lemma} \label{lemma:sigmachar} $\Pi$ is decidable iff there
is a computable function $\sigma$ such that $\varphi_{\sigma(x)}$ decides
$P(x)$ for all $x \in \N$. The decider $\pi$ for $\Pi$ and the function
$\sigma$ are computationally equivalent, in that either can be computed
from the other. \qed
\end{lemma}

If $\Pi$ is decidable, then every block of $\Pi$ is a decidable set. The
converse is not true. As Lemma~\ref{lemma:sigmachar} suggests, a
\emph{uniform} decider for the blocks is needed. An example of an
undecidable equivalence on $\N$ whose blocks are all recursive can be
manufactured from an undecidable problem in two variables which becomes
decidable if one of its variables is fixed. Consider the two-variable
decision problem, given a program $x$ and a (suitable encoding of) a
cardinal number $m \le \aleph_0$ to decide if $|W_x| = m$. Write
$C(x,m) = [|W_x| = m]$. By Rice's Theorem $C(x,m)$ is undecidable, however
if $x$ is fixed, it becomes trivial. Define an equivalence $\Pi$ as
\[
  \code{x,m}\Pi\code{x',m'} :\Leftrightarrow x = x' \wedge C(x,m) = C(x,m').
\]
The block $P(\code{x,m}, \Pi) = \{\code{x, m'} : C(x,m) = C(x,m')\}$ is a
recursive set as $C(x,m)$ is computable for fixed $x$. \emph{If} the
equivalence as a whole was decidable, we could decide $E(x,m,m') =
[C(x,m) = C(x,m')]$ uniformly in $x, m, m'$. Now let $x, m$ be given.
Choose two distinct cardinals $p, q$ which are also distinct from $m$.
At least two of $C(x,m), C(x,p), C(x,q)$ must be false. A complete
discussion of the cases

\begin{figure}[h!]
\begin{tabular}{c|c|c|c|c|c}
$C(x,m)$ & $C(x,p)$ & $C(x,q)$ & $E(x,m,p)$ & $E(x,m,q)$ & $E(x,p,q)$ \\
\hline
$\fr f$  & $\fr f$  & $\fr t$  & $\fr t$    & $\fr f$    & $\fr f$    \\
$\fr f$  & $\fr t$  & $\fr f$  & $\fr f$    & $\fr t$    & $\fr f$    \\
$\fr t$  & $\fr f$  & $\fr f$  & $\fr f$    & $\fr f$    & $\fr t$    \\
$\fr f$  & $\fr f$  & $\fr f$  & $\fr t$    & $\fr t$    & $\fr t$    \\
\end{tabular}
\end{figure}

\noindent characterises the statement $C(x,m) = \fr t$. Thus if we assume
$E(x,m,m')$ to be uniformly computable in all its parameters, we can decide
$|W_x| = m$ for any given $x$ and $m$, which is a \emph{contradiction}.

\begin{proposition} \label{prop:pchar} Let $A$ be a non-empty initial
segment, i.e.~either $A = \{0, \dots, n-1\}$ for some $n \in \N^+$ or
$A = \N$. Then a partition $\Pi = \{P_i : i \in A\}$ is decidable iff
there is a partial recursive function $p$ such that if $i \in A$ it follows
that $\varphi_{p(i)}$ decides $P_i$. A program for $p$ can be obtained from
a decider $\pi$ for $\Pi$ and vice versa.
\end{proposition}

\begin{proof} ``$\Rightarrow$'': Let $\pi$ decide $\Pi$. It follows from
Lemma~\ref{lemma:sigmachar} that every block is decidable. We need to
enumerate a decider for every block of $\Pi$ so that each block is represented
exactly once among the first $|\Pi|$ values of $p$. Define $i(x) :=
\mu x'[x\Pi x']$ which is total recursive via $\pi$ and returns the least
representative of the block of $x$. Let $\sigma$ be the function obtained
from $\pi$ by Lemma~\ref{lemma:sigmachar}. Now
\begin{align*}
  n(x) &:= \mu x'[\forall y < x: i(x') \not= in(y)], \\
  p(x) &:= \sigma n(x)
\end{align*}
are both computable by appeal to the Church-Turing Thesis. $p$ is obtainable
uniformly effectively from $\pi$. We will prove that $p$ is the wanted
enumeration. The $\mu$ search in $n$ for input $x$ computes $n(y)$ for all
$y < x$. It follows that if $p$ halts on input $x$ it must have halted on
all inputs $y < x$ and if $p$ does not halt on $y$ it will not halt on
inputs $x > y$. The set of inputs where $p$ halts is an initial segment $A$
of $\N$. If $x \in A$, $p(x) \in \rng \sigma$ is a program to decide some
block of $\Pi$. We claim that $n$ possesses these three properties:
\begin{compactenum}[\hskip 2em(a)]
\item $in(x) = n(x)$,
\item $x < z \Rightarrow n(x) < n(z)$ and
\item $x \le n(x)$
\end{compactenum}
for all $x, z$. Clearly $in(x) \le n(x)$ and furthermore $\forall y < x:
iin(x) = in(x) \not= in(y)$ where the inequality follows by definition of
$n(x)$. This means that $in(x)$ also satisfies the condition over which
$n(x)$ is minimised, so $n(x) \le in(x)$ and it must hold equality. For the
second property observe that the sets over which $n(x)$ and $n(z)$ minimise
are isotone: $\{x' : \forall y < x: i(x') \not= in(y)\} \subsetneq \{x' :
\forall y < z: i(x') \not= in(y)\}$. The difference of these sets includes
$n(x)$ so that indeed $n(x) < n(z)$. The last property follows by induction
from the monotonicity: we have $0 = n(0)$ and if $x \le n(x)$ for some
$x$, it follows $x \le n(x) < n(x+1)$ and thus $x+1 \le n(x+1)$.

We can now prove that each block $P \in \Pi$ has at least one decider in
$\rng p$: Let $j$ be the minimal element in $P$, then $i(j) = j$. If $j$ was
not already taken as an image $n(y)$ for $y < j$, it must hold that
$\forall y < j: j \not= n(y)$. But $j = i(j)$ and $n(y) = in(y)$, so that
$\forall y < j: i(j) \not= in(y)$, i.e.~$j$ satisfies the condition over
which $n(j)$ minimises and $j \ge n(j)$. By the third property above also
$n(j) \ge j$ and so $j = n(j)$ at the latest. On the other side, we see from
the definition of $n$ that $\forall y < x: in(x) \not= in(y)$, which means
$n(x)$ and $n(y)$ are in different blocks. Therefore $\sigma n(x)$ decides
a different block than $\sigma n(y)$ for $y < x$. Consequently each block
has at most one decider. In conclusion $p$ is partial recursive and halts
precisely on an initial segment of $\N$. Wherever it halts it gives a program
to decide a block of $\Pi$ so that each block gets exactly one decider.

``$\Leftarrow$'': The function $j(x) := \mu j[x \in P_j]$ is computable by
means of $p(j)$. It is a total function because $\Pi$ is a partition. $A$ is
an initial segment of $\N$ and every $x$ will be found in some $P_j$ for
$j \in A$, i.e.~before, in the $\mu$ search in $j(x)$, $p$ must be evaluated
for an argument $j \not\in A$, where its behaviour is not specified. Then
$x \Pi x' \Leftrightarrow j(x) = j(x')$. Hence $\pi(x, x') := [j(x) = j(x')]$
decides $\Pi$ and can be obtained uniformly effectively from a program for
$p$.
\end{proof}

Another way to constructively define equivalence relations is to require
a method which assigns each number a label. The equivalence classes are
formed by grouping equally labeled numbers together:

\begin{proposition} \label{prop:recdec} For each decidable equivalence
$\Pi$ there is a recursive function $r$ whose non-empty fibers are the
blocks of $\Pi$. Conversely the set of non-empty fibers of any recursive
function is a decidable equivalence. There are algorithms to convert a
decider for $\Pi$ into a recursive function and vice versa.
\end{proposition}

\begin{proof} If $\Pi$ is decidable, the function $r(x) := \mu x'[x \Pi x']$
is computable. Obviously $r(x) = r(x') \Leftrightarrow x \Pi x'$. If $r$ is
any recursive function, then $\lambda xx'[r(x) = r(x')]$ is a recursive
decider for the equivalence induced by the non-empty fibers of $r$.
\end{proof}

The results thus far enable us to represent a decidable partition $\Pi$ in
one of four constructive ways: \begin{inparaenum}[(1)] \item as the
recursive function $\pi$ which decides the predicate $[x \Pi x']$, \item as
the $\sigma$ function from Lemma~\ref{lemma:sigmachar}, \item as the
partial recursive $p$ from Proposition~\ref{prop:pchar} which enumerates
deciders for the blocks of $\Pi$, and \item as the non-empty fibers of a
recursive function\end{inparaenum}. The function $\pi$ decides $\equiv_\Pi$
as a set of ordered pairs while $p$ corresponds to the representation of the
partition $\Pi$ as a set of blocks. We have thus seen that these
representations are computationally equivalent, i.e.~we can go effectively
from one representation to another. This result justifies our identification
of partitions as sets of blocks with equivalence relations in the form of
sets of ordered pairs.

We remark in passing that it is not critical that $p$ enumerates deciders
for the blocks; recursive enumerators are sufficient. The proof is via dove
tailing.

\begin{proposition} \label{prop:qchar} Let $A$ be a non-empty initial
segment and $\Pi = \{P_i : i \in A\}$. Then $\Pi$ is decidable iff there is
a partial recursive $q$ such that $\dom q = A$ and if $i \in A$, then
$\varphi_{q(i)}$ enumerates $P_i$. A program for $q$ can be obtained from
a decider $\pi$ for $\Pi$ and vice versa. \qed
\end{proposition}

\paragraphtitle{Decidable cycles} \parlabel{par:decperm}
For any permutation $f$, the set of cycles of $f$ form a partition of its
domain. Denote the cycle of $f$ to which $x$ belongs by $[x]_f$. Let $\Phi$
temporarily denote the function which maps a permutation to its thus
associated partition. Conversely, to any partition $\Pi$ of a set, the
inverse image $\Phi^{-1}(\Pi)$ is the set of permutations associated with
it. For Recusion Theory, these considerations are restricted to decidable
equivalences and recursive permutations.

\begin{definition} If $\Pi$ is a decidable equivalence, the set of
associated permutations is $\Perm \Pi := \Phi^{-1}(\Pi) = \{f \in \G :
\forall x: P(x, \Pi) = [x]_f\}$. A recursive permutation $f$ has an
associated partition $\Part f := \Phi(f) = \{[x]_f : x \in \N\}$. The
corresponding equivalence relation, for infix usage, is denoted $\equiv_f$.
The central object of study in this paper is the set
\[
  \Perm := \bigcup_{\text{$\Pi$ decidable}} \Perm \Pi.
\]
\end{definition}

Writing the cycle equivalence relation down explicitly, $x \equiv_f x'
:\Leftrightarrow \exists k \in \Z: x' = f^k(x)$, it takes the form of a
\emph{reachability relation}. A recursive permutation $f$ is in $\Perm$
precisely when reachability in its undirected functional graph is
decidable.

It is in general not true that $\Part f$ is a decidable partition for each
$f$ or that $\Perm \Pi$ is non-empty for each decidable $\Pi$.
Section~\ref{sec:charperm} treats the questions of when these statements
do hold extensively. At this point we give the relevant definitions only:

\begin{definition} An equivalence $\Pi$ is \emph{permutable} if $\Perm \Pi
\not= \emptyset$. Conversely, a permutation $f$ has \emph{decidable cycles}
if $\Part f$ is a decidable equivalence.
\end{definition}

In this and the following subsection we concentrate on the relation of the
many possible permutations which belong to a single decidable partition
$\Pi$ and describe the conjugacy classes in $\Perm$ in terms of decidable
equivalences.

\begin{definition} In a subset of the group $\G$, two permutations $f$ and
$g$ are $\G$-conjugate if there is $h \in \G$ such that $f = h^{-1}gh$.
Instead of $\G$-conjugate we will use the term \emph{effectively conjugate}
or just \emph{conjugate}. This relation is denoted as $f \sim g$.
\end{definition}

A first step in understanding $\Perm \Pi$ is provided by

\begin{proposition} \label{prop:permconj1} Let $\Pi$ be decidable and
$f, g \in \Perm \Pi$. Then $f$ and $g$ are effectively conjugate. Moreover,
given programs for $f, g$ and a decider $\pi$ for $\Pi$, a program to
compute the conjugation between $f$ and $g$ can be described.
\end{proposition}

If $f$ and $g$, as above, belong to the same equivalence relation, they have
the same cycle type and would be conjugate in the full (non-constructive)
symmetric group of the natural numbers. Although $\G$ is not a symmetric
group, the standard proof, as found in \cite[Thm. 2.9]{bhatt98} carries
through under the additional premises of the proposition. A major
non-constructive step in their proof is to pair cycles from $f$ and $g$ with
matching lengths. This problem does not arise in our scenario because for
each $x$ we know that the cycle in $f$ to which $x$ belongs has the same
length as the cycle for $x$ in $g$, since $f$ and $g$ are permutations
associated to the same partition. It is also convenient that the cycles for
$x$ in $f$ and in $g$ contain the same elements. Before we give the proof,
we introduce the $\xi$ operator which is the $\mu$ operator equivalent for
(signed) integers. Its introduction is motivated by the observation that
powers of permutations act like integers modulo cycle length.

\begin{definition} \label{def:delta} \
\begin{enumerate}[\hskip .99em(a)]
\item Let $\delta: \Z \to \N$ denote the (intuitively computable) bijection
\[
  \delta(k) := \begin{cases}
    0, & k = 0, \\
    -2k - 1, & k < 0, \\
    2k, & k > 0.
   \end{cases}
\]
\item A partial function $\beta: \Z \rightsquigarrow \Z$ is called
\emph{partial recursive} if $\delta\beta\delta^{-1}: \N \rightsquigarrow
\N$ is a partial recursive function.
\item Let $p: \Z \rightsquigarrow \Z$ be a partial recursive function. By
writing $\xi k[p(k)]$ we mean that $k$ is an integer and we perform
\emph{alternating search}, that is we evaluate $p(0), p(-1), p(1), p(-2),
p(2), \dots$ in this order, until $p$ becomes $\fr t$ for the first time.
The first argument $k$ in this process achieving $p(k) = \fr t$ becomes
the value of the $\xi$ expression and if no such argument exists or $p$
is undefined on some argument on the way, $\xi$ does not terminate.
\end{enumerate}
\end{definition}

\begin{lemma} $p: \Z \rightsquigarrow \Z$ be a partial recursive function.
Then $\xi k[p(k)]$ is partial recursive.
\end{lemma}

\begin{proof} Observe that $\xi k[p(k)] =
\delta^{-1}\mu i[\delta p\delta^{-1}(i)]$ with the exact same semantics.
\end{proof}

The $\delta$ function takes an integer $k$ and, to produce the output
natural number, doubles the absolute value of $k$ and subtracts from this
the sign bit of $k$. To compute the inverse of $\delta$, examine the parity
of the input, as it determines the sign of the output. The absolute value of
the output integer is half of the input, rounded up to the next integer.
Based on these intuitions for dealing with $\delta$ and $\delta^{-1}$, and
by appeal to the Church-Turing Thesis, the function
$\lambda xx'\xi k[f^k(x) = x']$ is partial recursive and total on pairs
$(x, x')$ with $x \equiv_f x'$.

\begin{proof}[Proof of Proposition~\ref{prop:permconj1}] We construct the
conjugation $h$. Let $\pi$ decide $\Pi$ and $x$ be given, then calculate
$y(x) := \mu y[x \Pi y]$ via $\pi$. This calculation terminates as
$x \Pi x$, so $y(x) \le x$. By alternating search we determine $k(x) :=
\xi k[f^ky(x) = x] \in \Z$ which must exist. Return $h(x) := g^{k(x)}y(x)$.
The following facts are immediate from the definition of $h$ and the
prerequisites:
\begin{compactenum}[\hskip 2em(i)]
\item $y(z) = y(x)$ iff $z \in [x]_f$ iff $z \Pi x$ iff $z \in [x]_g$,    
\item for any $x$ the cycle lengths $|[x]_f|$ and $|[x]_g|$ match by (i), 
\item for any $x$ the value $y(x)$ is a fixed point of $h$,               
\item $f^ky(x) = f^{k'}y(x)$ iff $g^ky(x) = g^{k'}y(x)$ by (ii),          
\item if $z \Pi x$ then $h(z) = g^{k(z)}y(x)$ by (i).                     
\end{compactenum}

Because $f^{k(x)}y(x) = x$, it follows that
\[
  hf^{k(x)}y(x) = g^{k(x)}y(x)
\]
holds for all inputs $x$. We want to modify this equation so that the
exponent $k$ is free in $\Z$. Take any $x \in \N$ and $k \in \Z$ and set
$z = f^ky(x)$ so that $z \Pi x$. There is also the representation $z =
f^{k(z)}y(x)$ by (v). This yields $h(z) = g^{k(z)}y(x) = g^ky(x)$ by (iv),
which implies
\begin{align}
  hf^ky(x) = g^ky(x) \tag{$*$} \label{eq:hfkyx}
\end{align}
for all $x \in \N$ and $k \in \Z$.

To see that $h$ is a permutation, assume first $h(x) = h(x')$,
i.e.~$g^{k(x)}y(x) = g^{k(x')}y(x')$. Because $y(x) = g^{k(x')-k(x)}y(x')$,
$y(x)$ and $y(x')$ are in the same cycle of $g$ and hence in the same cycle
of $f$ and must therefore be equal, as values in the range of $y$. We have
shown $g^{k(x)}y(x) = g^{k(x')}y(x)$ which implies $x = f^{k(x)}y(x) =
f^{k(x')}y(x) = x'$. To show surjectivity, let $z$ be given. There is a
representation as $z = g^ky(z)$ for some $k \in \Z$ because $z \in [y(z)]_f
= [y(z)]_g$. But then already $hf^ky(z) = g^ky(z) = z$ by~\eqref{eq:hfkyx}.
This shows that $h$ is a permutation, and for all $x = f^ky(x)$ we see
via~\eqref{eq:hfkyx} that
\[
  h^{-1}gh(x) = h^{-1}ghf^ky(x) = h^{-1}g^{k+1}y(x) = f^{k+1}y(x) = f(x),
\]
which completes the proof.
\end{proof}

\paragraphtitle{Isomorphism and conjugacy} \parlabel{par:isoconj}
Proposition~\ref{prop:permconj1} shows that two permutations with identical
equivalences are conjugate. The converse is not true as the following simple
example shows: let $f = \cycle{\dots, 6, 2, 0, 4, 8, \dots}$ and
$g = \cycle{\dots, 7, 3, 1, 5, 9, \dots}$. They define different
equivalences but are conjugate via $\cycle{0,1}\cycle{2,3}\cycle{4,5}\dots$.
However, the equivalences defined by $f$ and $g$ are ``essentially the
same'', in that they have the same block structure. This observation
suggests that studying equivalence isomorphisms leads to a better
understanding of effective conjugacy classes.

\begin{definition} $\Pi$ and $\Gamma$ be equivalences. A recursive
permutation $\theta$ is a \emph{$(\Pi, \Gamma)$-isomorphism} if $x \Pi y
\Leftrightarrow \theta(x) \Gamma \theta(y)$. Isomorphy of $\Pi$ and $\Gamma$
is written as $\Pi \cong \Gamma$.
\end{definition}

The next goal is a characterisation for the solvability of conjugation
equations $f = h^{-1}gh$ when $f, g$ have decidable cycles, and an algorithm
for computing the solution $h$ if a witness for the solvability is
available. This result extends Proposition~\ref{prop:permconj1}. To this
end, some preparations are needed.

\begin{lemma} Let $\Pi$ be a decidable equivalence, $h$ a recursive
permutation. Then $\Pi^h := h(\Pi) := \{h(P) : P \in \Pi\}$ is a decidable
partition.
\end{lemma}

\begin{proof} Since $h$ is a permutation, $\Pi^h$ is again a partition.
Then $x \Pi^h x'$ iff $P(x, \Pi^h) = P(x', \Pi^h)$, which is the case iff
$x$ and $x'$ belong to the same $h(P)$, for some $P \in \Pi$. But this is
equivalent to $h^{-1}(x)$ and $h^{-1}(x')$ belonging to the same
$P \in \Pi$, i.e.~$h^{-1}(x) \Pi h^{-1}(x')$. Thus
$\lambda xx'[h^{-1}(x) \Pi h^{-1}(x')]$ is a recursive decider for $\Pi^h$.
\end{proof}

\begin{lemma} \label{lemma:fpi} Let $f$ be any recursive permutation. Then
$x \Pi x' \Leftrightarrow f(x) \Pi^f f(x')$.
\end{lemma}

\begin{proof} $\Pi^f$ is a decidable partition. Then clearly $x \Pi x'
\Rightarrow f(x) \Pi^f f(x')$. The converse follows by repeating the
argument with $f^{-1}$ in place of $f$.
\end{proof}

\begin{proposition} \label{prop:gammathetapi} If $\theta$ is a recursive
permutation, then it is a $(\Pi, \Gamma)$-isomorphism iff $\Gamma =
\Pi^\theta$.
\end{proposition}

\begin{proof} ``$\Rightarrow$'': If $\theta$ is an isomorphism, we have
$y \Gamma y' \Leftrightarrow \theta^{-1}(y) \Pi \theta^{-1}(y')
\Leftrightarrow y \Pi^\theta y'$ according to Lemma~\ref{lemma:fpi}. Thus
$\Gamma = \Pi^\theta$.

``$\Leftarrow$'': $x \Pi x' \Leftrightarrow \theta(x) \Pi^\theta \theta(x')
\Leftrightarrow \theta(x) \Gamma \theta(x')$.
\end{proof}

The set of $(\Pi, \Pi)$-isomorphisms is the \emph{automorphism group} of
$\Pi$, denoted $\Aut \Pi$. This is indeed a group under composition (a
subgroup of $\G$) as $x \Pi x' \Leftrightarrow \id(x) \Pi \id(x')$, hence
$\id \in \Aut \Pi$. If $f, g \in \Aut \Pi$, then by
Proposition~\ref{prop:gammathetapi}: $\Pi^{fg} = (\Pi^g)^f = \Pi^f = \Pi$
and $\Pi^{f^{-1}} = (\Pi^f)^{f^{-1}} = \Pi^{f^{-1}f} = \Pi$ so that
$fg, f^{-1} \in \Aut \Pi$.

\begin{lemma} \label{lemma:fxequivx'} Let $\Pi$ be decidable and $f \in \G$.
Consider the predicate
\begin{align}
  \forall x, x' : x \Pi x' \Leftrightarrow f(x) \Pi x'. \tag{$*$} \label{eq:fxequivx'}
\end{align}
Then $f \in \Perm \Pi \Rightarrow \eqref{eq:fxequivx'} \Rightarrow
f \in \Aut\Pi$. In particular $\Perm \Pi \subseteq \Aut \Pi$.
\end{lemma}

\begin{proof} If $f \in \Perm \Pi$, it follows that $x \Pi f(x)$. Together
with $x' \Pi x$, this implies $x' \Pi f(x)$. With $f \in \Perm \Pi$ is also
$f^{-1} \in \Perm \Pi$ and thus $f(x) \Pi x' \Rightarrow
x = f^{-1}f(x) \Pi x'$.

Now assume $f \in \G$ and \eqref{eq:fxequivx'} holds. By symmetry of $\Pi$
and twofold application of \eqref{eq:fxequivx'}, we obtain $x \Pi x'
\Leftrightarrow f(x) \Pi f(x')$, i.e.~$\Pi = \Pi^f$ by
Lemma~\ref{lemma:fpi}. Then it follows that $f \in \Aut \Pi$ by
Proposition~\ref{prop:gammathetapi}.
\end{proof}

With Lemma~\ref{lemma:fxequivx'} we are in a position to approach the
following problem: given $f \in \Perm \Pi$, $\Pi$ decidable, and $h$ a
recursive permutation, a member of $\Perm \Pi^h$ is to be described. If we
assume that there is some $g \in \Perm \Pi^h$, it would satisfy the relation
$g(y) \Pi^h y' \Leftrightarrow y \Pi^h y' \Leftrightarrow
h^{-1}(y) \Pi h^{-1}(y') \Leftrightarrow fh^{-1}(y) \Pi h^{-1}(y')
\Leftrightarrow hfh^{-1}(y) \Pi^h y'$, by Lemmata~\ref{lemma:fxequivx'}
and~\ref{lemma:fpi}. This shows that, at least under the assumption that
$\Pi^h$ is permutable, $hfh^{-1} \in \Aut \Pi^h$. The following proposition
strengthens this deduction in two ways: indeed $hfh^{-1} \in \Perm \Pi^h$
and this is independent of the assumption that $\Pi^h$ is permutable. A
technical lemma is needed, which is obvious from the definition of $\Perm
\Pi$:

\begin{lemma} \label{lemma:permcrit} $f \in \Perm \Pi$ iff $f \in \G$ and
for all $x_0$ the function $\Z \ni k \mapsto f^k(x_0)$ is surjective on
$P(x_0, \Pi)$. \qed
\end{lemma}

\begin{proposition} \label{prop:memberhpi} Let $f \in \Perm \Pi$ and $h \in
\G$, then $hfh^{-1} \in \Perm \Pi^h$.
\end{proposition}

\begin{proof} Let $y_0 = h(x_0)$ and $y = h(x) \in P(y_0, \Pi^h)$ be given.
Then $y_0 \Pi^h y$ implies $x_0 \Pi x$ by Lemma~\ref{lemma:fpi}. Because
$f \in \Perm \Pi$, we can find $k$ so that $f^k(x_0) = x$. Now
$(hfh^{-1})^k(y_0) = hf^kh^{-1}(y_0) = hf^k(x_0) = h(x) = y$. It remains to
show that $hf^kh^{-1}(y_0)$ is in $P(y_0, \Pi^h)$ for every $k$. Because
$f \in \Perm \Pi$, it is $f^k(x_0) \in P(x_0, \Pi)$ and thus $hf^k(x_0) \in
P(y_0, \Pi^h)$ by definition of $\Pi^h$. Lemma~\ref{lemma:permcrit} shows
that $hfh^{-1} \in \Perm \Pi^h$.
\end{proof}

We remark that a more satisfying proof of this proposition can be given with
the notions introduced in \parref{par:orderperm}. It is not difficult to
prove that the condition \eqref{eq:fxequivx'} is characteristic for the set
$\I\Pi$ which lies between $\Perm \Pi$ and $\Aut \Pi$ and is equipped with a
preorder in \parref{par:orderperm}. The mapping $f \mapsto hfh^{-1}$ is an
order isomorphism $\I\Pi \to \I\Pi^h$ and it is shown that $\Perm \Pi$ is
the set of upper bounds on $\I\Pi$. Since an order isomorphism preserves
upper bounds, it follows that $hfh^{-1}$ is an upper bound on $\I\Pi^h$
which immediately yields $hfh^{-1} \in \Perm \Pi^h$.

We have remarked before the proof of Proposition~\ref{prop:permconj1} that
conveniently the permutations $f$ and $g$ for a decidable partition $\Pi$
have not only the same cycle type and that the cycles are paired already
according to their length but the cycles of $f$ and $g$ to which an $x$
belongs contain the same elements. Replacing identity of cycles with
isomorphy yields a characterisation of conjugacy for members of $\Perm$:

\begin{theorem} \label{theorem:permconj2} Let $f, g$ be recursive
permutations for decidable partitions $\Pi, \Gamma$ respectively. Then $f$
and $g$ are effectively conjugate iff $\Pi$ and $\Gamma$ are isomorphic. Put
another way:
\[
  \forall f, g \in \Perm: f \sim g \Leftrightarrow \Part f \cong \Part g.
\]
If a decider for one of the equivalences is known, a conjugation between $f$
and $g$ can be obtained uniformly effectively from an isomorphism of the
equivalences and vice versa.
\end{theorem}

\begin{proof}
``$\Rightarrow$'': Let $f = h^{-1}gh$ for some recursive permutation $h$.
We claim that $h$ is a $(\Pi, \Gamma)$-isomorphism. By
Proposition~\ref{prop:gammathetapi} this is equivalent to the statement
$\Pi^h = \Gamma$. By Proposition~\ref{prop:memberhpi} we have $f = h^{-1}gh
\in \Perm \Gamma^{h^{-1}}$, but also $f \in \Perm \Pi$. Thus the blocks of
$\Pi$ are the cycles of $f$ which are also the blocks of $\Gamma^{h^{-1}}$.
This implies $\Pi = \Gamma^{h^{-1}}$ and since $h$ is a permutation,
$\Pi^h = \Gamma$.

``$\Leftarrow$'': Let $\theta$ be a
$(\Pi, \Gamma)$-isomorphism. Then $\theta^{-1}g\theta$ is in $\Perm \Pi$
because by Proposition~\ref{prop:memberhpi}: $\theta^{-1}g\theta \in \Perm
\Gamma^{\theta^{-1}}$ and $\Gamma^{\theta^{-1}} = \Pi$ by
Proposition~\ref{prop:gammathetapi}. By virtue of
Proposition~\ref{prop:permconj1} there is a recursive permutation $h$ such
that $f = h^{-1}\theta^{-1}g\theta h$, thus $f$ and $g$ are effectively
conjugate via $\theta h$.
\end{proof}

In view of permutations $f, g \in \Perm$, the isomorphy of their respective
equivalences $\Part f \cong \Part g$ is also called \emph{effective cycle
type equality}. From Theorem~\ref{theorem:permconj2} it follows that if $f$
has decidable cycles and $\Pi$ is decidable and permutable, then
\begin{gather*}
  [f]_\sim = \bigcup_{\Gamma \cong \Part f} \Perm \Gamma, \\
  [\Pi]_{\cong} = \{\Part g : \exists f \in \Perm \Pi : g \sim f\},
\end{gather*}
and using the former equation
\begin{align*}
  \Perm = \bigcup_{\Pi \; \text{dec.}} \Perm \Pi
        = \bigcup_{\Part f \; \text{dec.}} [f]_\sim.
\end{align*}
Since $\Perm$ is a union of conjugacy classes, we obtain

\begin{corollary} \label{coro:permnormal} $\Perm$ is a normal subset of the
group of recursive permutations. \qed
\end{corollary}

\section{Characterisations of $\Perm$} \label{sec:charperm}

The definition of $\Perm$ establishes a relation between the permutations of
$\N$, $\Sym \N$, and the equivalence relations over $\N$, $\Eqv \N$.
Regarding the properties of ``recursiveness'' and ``decidability'', this
relation is non-trivial, as each of the four cases in
Figure~\ref{fig:symeqv} is possible. From this relation arise two
approaches to characterise $\Perm$. The first starts in the ``recursive
permutation'' column and asks when the permutation has decidable cycles and
the second starts in the ``decidable equivalence'' row asks when the
equivalence is permutable. The global questions whether every recursive
permutation has decidable cycles and whether all decidable equivalences are
permutable are both answered in the negative, as referenced in
Figure~\ref{fig:symeqv}. This section follows these two directions and
also gives an order-theoretic characterisation of $\Perm$.

\begin{figure}
\bgroup
\def\arraystretch{1.3}
\begin{tabular}{c|c|c}
\diagbox{$\Eqv \N$}{$\Sym \N$} & recursive & non-recursive \\
\hline
decidable & $\Perm$ & Proposition~\ref{prop:nonpermut} \\
\hline
undecidable & Lemma~\ref{lemma:gminusperm} & $\Pi = \{K, \overbar K\}$
\end{tabular}
\egroup
\caption{All possibilities for the relation of recursiveness of
permutations and decidability of equivalences are realisable. The
equivalence $\Pi = \{K, \overbar K\}$, with $K$ the Halting Problem, is
undecidable because its blocks are undecidable. Any permutation which has
$\overbar K$ as a cycle is not recursive because the cycles of a recursive
permutation are recursively enumerable.}
\label{fig:symeqv}
\end{figure}

\paragraphtitle{Cycle decidability} \parlabel{par:cycledec} This subsection
gives cycle decidability criteria, i.e.~how must a recursive permutation
be designed so that its cycles induce a decidable equivalence. There is a
sufficient condition which is practically useful because it only requires
some information from the cycle type of $f$:

\begin{lemma} \label{lemma:oneinfdec} If $f$ is a recursive permutation with
at most one infinite cycle, then $\Pi = \Part f$ is decidable. A decider for
$\Pi$ can be computed from $f$.
\end{lemma}

\begin{proof} Let $x, x'$ be given. If $x = x'$, the case is clear and we
report $\pi(x, x') := \fr t$. In the other case, search
\[
  i = \mu i[i \ge 1 \wedge \{f^i(x), f^i(x')\} \cap \{x, x'\} \not= \emptyset].
\]
This search always terminates:
\begin{enumerate}[(i)]
\item If both $x$ and $x'$ are in an infinite cycle, they must be in the
same cycle. Because $x \not= x'$ at this point, there is an $i \ge 1$ such
that $f^i(x) = x'$ or $f^i(x') = x$, whereas $f^i(x) = x$ or $f^i(x') = x'$
is impossible for $i \ge 1$.
\item If they are in the same finite cycle, we will find $f^i(x) = x'$ or
$f^i(x') = x$ \emph{before} $f^i(x) = x$ or $f^i(x') = x'$ can be
encountered.
\item If they are in different cycles $f^i(x) = x' \vee f^i(x') = x$ is
unsatisfiable. But at least one of them must be in a finite cycle and thus
$f^i(x) = x$ or $f^i(x') = x'$ will be found.
\end{enumerate}
These are all the cases. If $f^i(x) = x'$ or $f^i(x') = x$, report
$\pi(x, x') := \fr t$ and else $\pi(x, x') := \fr f$. The function $\pi$
decides the cycles of $f$ and can be computed uniformly effectively from
$f$.
\end{proof}

Lemma~\ref{lemma:oneinfdec} can be generalised to the case of finitely many
infinite cycles, but the constructed decider is not uniform in $f$ anymore.
The construction requires a system of representatives for the infinite
cycles. For this reason, it is given as a separate proposition. This result
can already be found in Myhill's paper \cite[Cor.~1]{myhill59} on
\emph{splinters} in digraphs of recursive functions, which generalise cycles
in permutations. Myhill's corollary implies that all cycles of a permutation
are recursive sets if there are only finitely many infinite cycles. As seen
in \parref{par:decequiv}, recursivity of all blocks of a partition does not
suffice for (uniform) decidability of the partition. That addition, however,
follows easily from his Theorem~1.3 and our characterisation of cycle
decidability via transversals in Theorem~\ref{theorem:cycledecchar} below.
We give an independent proof here:

\begin{proposition} \label{prop:fininfdec} If $f$ is a recursive permutation
with finitely many infinite cycles, then $\Pi = \Part f$ is decidable. Using
the definition from the introduction: $\G_1 \subseteq \Perm$.
\end{proposition}

\begin{proof} Let $x_1, \dots, x_n$ be a \emph{system of representatives}
for the infinite cycles, i.e.~each $x_i$ belongs to an infinite cycle, no
two of them belong to the same, and every infinite cycle has a
representative among the $x_i$'s. Given $x, x'$ we calculate indices $k,
k'$ respectively. If $x = x'$, set $k = k' = 0$. If $x = x_i$, then set
$k = 0$, and analogously for $x'$ and $k'$. Otherwise $k$ is determined by
alternating search of $f^k(x)$ in the set $\{x, x', x_1, \dots, x_n\}$:
$k = \xi k[k \not= 0 \wedge f^k(x) \in \{x, x', x_1, \dots, x_n\}]$ and
$k'$ analogously by searching for $f^{k'}(x')$ in the same set. This
procedure terminates because $x$ is either in a finite cycle (and will
find itself eventually with an index $k \not= 0$) or in an infinite cycle
(and will find one of the representatives for infinite cycles with an index
$k \not= 0$).

Once $k, k'$ are computed, let $y = f^k(x), y' = f^{k'}(x')$ be the found
elements. Note that $y, y' \in \{x, x', x_1, \dots, x_n\}$. If $y = x'$ or
$y' = x$, then $x$ and $x'$ are in the same cycle. If $y = x \not= x'$ or
$y' = x' \not= x$, then one of $x, x'$ is in a finite cycle different from
the other's. Otherwise $y = x_i$ and $y' = x_j$ for $1 \le i, j \le n$.
Then $x \Pi x' \Leftrightarrow x_i = x_j$. This gives a way to decide $\Pi$.
\end{proof}

Kent's \cite[Thm.~1.7]{kent62}, as quoted in the introduction, states that
for the permutations with finitely many infinite cycles, cycle type equality
and effective conjugacy are equivalent. By Theorem~\ref{theorem:permconj2},
in turn, effective conjugacy in $\Perm$ and effective cycle type equality
are equivalent. The preceding Proposition~\ref{prop:fininfdec} yields that
$\G_1$ is contained in $\Perm$ and it follows

\begin{corollary} \label{coro:cycletypeeq} In $\G_1$ cycle type equality and
effective cycle type equality are the same. \qed
\end{corollary}

For characterisations of cycle decidability, in contrast to
Lemma~\ref{lemma:oneinfdec} and Proposition~\ref{prop:fininfdec}, much more
detailed information about the cycles is required.

\begin{theorem} \label{theorem:cycledecchar} Let $f$ be a recursive
permutation. The following conditions are equivalent:
\begin{compactenum}[\hskip .99em(a)]
\parbox{0.9\textwidth}%
{
  \item $\exists \pi: \pi(x,x') = [x \equiv_f x']$ \hfill(cycle decidability)
  \item $\exists \mu: \mu(x) = \mu x'[x \equiv_f x']$ \hfill(smallest cycle
    representative)
  \item $\exists \vartheta: (x \equiv_f x' \Rightarrow \vartheta(x) =
    \vartheta(x')) \wedge \vartheta(x) \equiv_f x$ \hfill(unique cycle
    representative)
  \item $\exists \chi: \chi(x) = \chi(x') \Leftrightarrow x \equiv_f x'$
    \hfill(characteristic value of a cycle)
  \item $\exists \rho: (\rho(e) \equiv_f \rho(e') \Rightarrow \rho(e) =
    \rho(e')) \wedge \forall x \exists e: x \equiv_f \rho(e)$
    \hfill(transversal)
}
\end{compactenum}
where all functions are taken to be total recursive. These functions are
pairwise computationally equivalent.
\end{theorem}

The condition (e) describes a recursively enumerable system of
representatives for all cycles. Such a system is called a \emph{choice set}
by \cite{kent62} and \cite{myhill59}, and a \emph{transversal} by
\cite{higman90}. We adopt the term ``transversal''. Higman furthermore calls
the transversal consisting of the smallest elements of each cycle
\emph{principal}, which corresponds to condition (b) above. As the theorem
shows, any recursively enumerable transversal of a permutation is
computationally equivalent to the principal transversal.

\begin{proof} We will show (a) $\Rightarrow$ (b) $\Rightarrow$ (e)
$\Rightarrow$ (a) and (b) $\Rightarrow$ (c) $\Rightarrow$ (d) $\Rightarrow$
(b).

``(a) $\Rightarrow$ (b)'': $\mu(x) := \mu x'[\pi(x, x')]$ is computable.

``(b) $\Rightarrow$ (e)'': We need to enumerate representatives of all the
cycles without representing a cycle twice with different values (it is
allowed to repeat values because $f$ may have only finitely many cycles).
This can be done by $\rho = \mu$.

``(e) $\Rightarrow$ (a)'': Given $x, x'$, start dove-tailing computations
of $f^k\rho(e)$ where $e$ varies in $\N$ and $k$ varies in $\Z$. Since
$\rho$ enumerates a complete system of representatives we will eventually
find $e, e'$ and $k, k'$ such that $f^k\rho(e) = x$ and $f^{k'}\rho(e') =
x'$. Because cycles are uniquely represented in $\rho$, it holds $x \equiv_f
x' \Leftrightarrow e = e'$.

``(b) $\Rightarrow$ (c)'': $\vartheta = \mu$ is possible.

``(c) $\Rightarrow$ (d)'': $\chi = \vartheta$ is possible because $x
\equiv_f x' \Rightarrow \vartheta(x) = \vartheta(x')$, and since $x \equiv_f
\vartheta(x)$, we have $x \not\equiv_f x' \Rightarrow \vartheta(x) \not=
\vartheta(x')$.

``(d) $\Rightarrow$ (b)'': Take $\mu(x) := \mu x'[\chi(x) = \chi(x')]$.
\end{proof}

\cite[Thm.~1.4]{myhill59} shows in the more general context of digraphs of
recursive functions (not necessarily permutations), but with the same
technique as above, that a recursively enumerable transversal for all weakly
connected components, which are the corresponding generalisation of cycles,
implies decidability of every component.

Aside from the transversal criterion, all characterisations of cycle
decidability listed in Theorem~\ref{theorem:cycledecchar} come in the form
of an external function which answers questions about the cycles of the
permutation. The next subsection follows a more profound approach to cycle
decidability. We define a normal form for permutations, an intrinsic
property of the permutation, and prove that conjugacy to a normal form
permutation is equivalent to cycle decidability.

\paragraphtitle{Normal forms} \parlabel{par:normal} Let $f$ be a recursive
permutation. Then for every block $P \in \Part f$ and arbitrary $x_P \in P$,
the function $\lambda j[f^{\delta^{-1}(j)}(x_P)]$ enumerates $P$. Thus
$\Part f$ is a partition whose blocks are recursively enumerable by the
powers of a single fixed permutation at appropriately chosen starting points.
Proposition~\ref{prop:pchar} shows that if $\Part f$ is decidable, every
block of it must be decidable. It is well-known that a set is recursive iff
it is recursively enumerable in non-decreasing order. We use this idea to
define a normal form for permutations with decidable cycles and derive a
characterisation for permutable decidable equivalences. The following lemma
provides an important technique for constructing a cycle out of an infinite
recursive set. Some variants of the idea will also be used later.

\begin{lemma} \label{lemma:buildcycle} Let $p$ be the characteristic
function of an infinite recursive set $A$. Then we can find a cyclic
permutation whose support is $A$.
\end{lemma}

\begin{proof} Let $x$ be given. $[x \in A]$ is computable by $p$. If
$x \not\in A$, then return $f(x) := x$. Else list the members of $A$ in
increasing order:
\begin{align*}
  a(0) &:= \mu x[x \in A], \\
  a(n+1) &:= \mu x[x > a(n) \wedge x \in A]
\end{align*}
which is uniform in $p$ and total recursive because $A$ is infinite.
Then define
\[
  f(x) := a\delta(\delta^{-1}a^{-1}(x)+1), \, x \in A.
\]
The $\delta$ function is indeed used here according to
Definition~\ref{def:delta}: Denote by $\succf$ the successor function
$\succf(x) = x+1$ which is a recursive permutation on $\Z$. Then
$f = a\delta\succf\delta^{-1}a^{-1}$ where $\delta\succf\delta^{-1}$
is a recursive permutation $\N \to \N$. This definition of $f$ takes members
of $A$ into $A$ injectively, as it is a composition of injective functions.
For surjectivity it suffices to observe that $a$ is surjective on $A$, and
that $\delta\succf\delta^{-1}a^{-1}$ maps $A$ onto $\N$. It is clear that
$f$ has no fixed points in $A$.
\end{proof}

In the normal form for cyclic permutations the structure of the cycle
(repeated application of the permutation or its inverse) shall exhibit
the increasingly ordered list of all members of the cycle, in an
algorithmically recognisable way. A first attempt would be to take all
members of the cycle in increasing order, $a_0 < a_1 < a_2 < \dots$ and
define the cycle to be $\cycle{a_0, a_1, a_2, \dots}$. This layout is easy
to work with but yields a permutation only if the number of $a_i$'s is
finite, for if not $a_0$ will not get an inverse image. For infinite cycles,
the idea of using $\delta$ from the proof of Lemma~\ref{lemma:buildcycle}
comes into play:

\begin{definition} A (finite or infinite) cycle $f = \cycle{\dots, a_{-1},
a_0, a_1, \dots}$ of length $n \in \N^+ \cup \{\infty\}$ whose least element
is $a_0$ is in \emph{normal form} if the function
$\lambda j[f^{\delta^{-1}(j)}(a_0)]$ is increasing for $0 \le j < n$. A
recursive permutation is in \emph{normal form} if every cycle in its
disjoint cycle decomposition is in normal form. Such permutations are
more briefly called \emph{normal}.

A permutation where every infinite cycle is normal and every finite cycle
$\cycle{a_0, \dots, a_{n-1}}$ with $a_0$ as smallest element, has
$\lambda j[f^j(a_0)]$ increasing for $0 \le j < n$, is called
\emph{semi-normal}.

Because there is only one way to order the entirety of a cycle into an
increasing sequence, there is for any permutation $f$ precisely one
(possibly non-recursive) normal permutation $f'$ with $\Part f = \Part f'$.
This $f'$ is called the \emph{normal form} of $f$. The \emph{semi-normal
form} of a permutation is defined analogously and also unique.
\end{definition}

The less straightforward definition of normality serves the purpose to unify
the look of normal cycles regardless of whether they are finite or infinite,
which cannot be decided, as is shown in \parref{par:cyclefinite}.
Indeed the idea from Lemma~\ref{lemma:buildcycle} can also be made to work
with finite cycles if their length is known, as Figure~\ref{fig:normaldiag}
indicates. An archetypal construction of a normal cycle already appeared in
Lemma~\ref{lemma:buildcycle}:

\begin{corollary} \label{coro:buildcyclenormal} The cycle constructed in the
proof of Lemma~\ref{lemma:buildcycle} is normal.
\end{corollary}

\begin{proof} In the nomenclature of the lemma, $f|_A :=
a\delta\succf\delta^{-1}a^{-1}$. The assertion follows by mere
calculation:
\begin{align*}
    f|_A^{\delta^{-1}(j)}a(0)
 &= (a\delta\succf\delta^{-1}a^{-1})^{\delta^{-1}(j)}a(0) \\
 &= a\delta\succf^{\delta^{-1}(j)}\delta^{-1}a^{-1}a(0) \\
 &= a\delta\delta^{-1}(j)
  = a(j),
\end{align*}
which is increasing. Outside of $A$ are only fixed points of $f$ which are
normal cycles.
\end{proof}

\begin{figure}
\input{normalcycle.tex}
\caption{Given a cycle $\cycle{a_0, a_1, \dots, a_n}$ with the rearrangement
$a_{\rho(0)} < a_{\rho(1)} < \dots < a_{\rho(n)}$ of the $a_i$'s into an
increasing sequence, use the $\delta$ technique from
Lemma~\ref{lemma:buildcycle} to obtain the displayed normal cycle. The
curved arrows indicate the increasing order of the elements, the order from
left to right is the cyclic order. It is assumed that $n$ is even for the
sake of the example.}
\label{fig:normaldiag}
\end{figure}
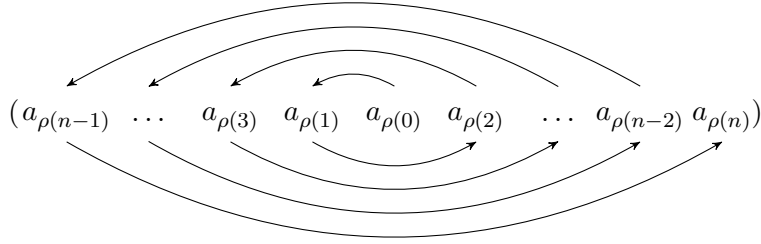

\begin{lemma} \label{lemma:finitesetcycle} There is an algorithm which is
uniform in $f \in \G$ and decides for every finite set $A$, given as a list
of its elements, if it is the union of cycles of $f$, and, if it is not,
computes a member of $[A]_f \setminus A$, where $[A]_f :=
\bigcup_{x \in A}  [x]_f$.
\end{lemma}

\begin{proof} We use the fact that a finite set is a union of cycles of $f$
iff it is closed under application of $f$. Obtain a list of the members of
$X = A \cup f(A)$, which is possible because $A$ was given as such a list.
Without loss of generality, these finite lists are without repetition. Then
we can decide if $|X| = |A|$. If this is the case, then $A$ is closed under
application of $f$ and thus a union of cycles. Otherwise $|X| > |A|$ and
$A$ is not a union of cycles. Since evidently $X \subseteq [A]_f$, we can
find an element in $[A]_f \setminus A$ by inspecting the list difference
$X \setminus A$.
\end{proof}

The encoding of a finite set as a finite bitstring where the $x$-th bit is
set iff $x$ is a member of the set is called \emph{canonical}
\cite[\textsection~5.6]{rogers87}. It is easily seen that this canonical
encoding is computationally equivalent to the encoding as a finite list.
This encoding was essential in the proof to obtain the cardinality of $A$.
It is shown in \cite[\textsection~5.6, Theorem~XV(b)]{rogers87} that the
cardinality of a finite set cannot be computed from a decider for that set,
which would have been another candidate for an encoding of finite sets.

\begin{proposition} \label{prop:existnormal} Let $f$ be a recursive
permutation and let $\pi$ decide $\Pi = \Part f$. Then there is a recursive
$f'$ in normal form with $\Part f' = \Part f$. A program for $f'$ can be
computed uniformly effectively from $f$ and $\pi$.
\end{proposition}

\begin{proof} To given $x$ find \emph{all} $x_0 < x_1 < \dots < x_k = x$
with $x_i \Pi x$. Call the list of these elements $A$. Using
Lemma~\ref{lemma:finitesetcycle}, we can determine if $A$ is a union of
cycles. If it is, it must be a single cycle because all members of $A$ are
in the same cycle. Then the appropriate image of $x = x_k$ can easily be
determined, according to Figure~\ref{fig:normaldiag}.

Otherwise Lemma~\ref{lemma:finitesetcycle} gives an element $x^* \in [A]_f
\setminus A = [x]_f \setminus A$. Since the $x_0, \dots, x_k$ are all
members of the cycle satisfying $x_i \le x$, it must be $x^* > x$. Search
$\hat x = \mu \hat x[\hat x > x \wedge \hat x \Pi x]$; this search will
terminate as $x^* \Pi x$ and thus $\hat x \le x^*$.

This algorithm provides a way to decide if there are still greater elements
than $x$ in its cycle and, in the affirmative case, to find the smallest
such element. Figure~\ref{fig:normaldiag} suggests that this is enough
information to carry the construction of the cycle out to its end, in the
finite as well as in the infinite case:
\[
  f'(x) := \begin{cases}
    x_{k+2}, & \text{$k$ even and $x_{k+2}$ exists}, \\
    x_{k+1}, & \text{$k$ even, $x_{k+1}$ exists and $x_{k+2}$ does not exist}, \\
    x_{k-1}, & \text{$k$ even and $x_{k+1}$ does not exist}, \\
    x_0, & \text{$k = 1$}, \\
    x_{k-2}, & \text{$k$ odd and $k > 1$}. 
   \end{cases}
\]
\par \vspace{-1.55\baselineskip}
\qedhere
\end{proof}

\begin{corollary} \label{coro:hasuniquenormal} If $\Pi$ is decidable and
permutable, then $\Perm \Pi$ contains exactly one permutation in normal
form. \qed
\end{corollary}

Let $f$ be a recursive permutation with decidable cycles. Then $\Part f$
contains, by Corollary~\ref{coro:hasuniquenormal}, a recursive permutation
$f'$ in normal form, which is the normal form of $f$.
Proposition~\ref{prop:permconj1} shows that $f$ is effectively conjugate
to $f'$. The converse is also true. Suppose that $f$ is effectively
conjugate to its normal form $f'$. It is to be shown that $\Part f'$ is
decidable. The proof exploits the resemblance between a normal cycle and a
strictly convex function: the smallest element $x_0$ of a cycle of $f'$ is
characterised by the condition $x_0 \le \min \{x_0^{\pm}\}$, where
$x^{\pm} := f^{\pm 1}(x)$. Thus, given any $x$, an alternating search
through the cycle, beginning at $x$, can be performed to find the smallest
element of $x$'s cycle:
\[
  \mu x_0[x_0 \equiv_{f'} x] = \xi f'^k(x)[f'^k(x) \le
    \min \{f'^{k+1}(x), f'^{k-1}(x)\}],
\]
where $\xi f'^k(x)[p(k)]$ is a more suggestive notation for
$f'^{\xi k[p(k)]}(x)$.

The alternating search can be replaced by a more intelligent algorithm which
further uses the resemblance of normal cycles and strictly convex functions,
namely that the direction from any point towards the minimum can be
determined by inspecting a neighbourhood of the given point. For any $x$,
compute $x^+$ and $x^-$. If, by the test above, $x$ is not the minimum of
the cycle, at least one of $x^+$ and $x^-$ must be smaller than $x$. Take
the smallest of both. If it is $x^+$, the minimum can be found in
$f$-positive direction, i.e.~the smallest element is of the form $f^k(x)$
with $k > 0$, and if the smaller element is $x^-$, the minimum is in
$f$-negative direction. By a variant of binary search, the minimum can
be found with a number of steps logarithmic in the distance from the
starting point to the minimum in the cycle.

\begin{lemma} \label{lemma:normalmin} The function
$\lambda x[\mu x'[x' \equiv_{f'} x]]$ can be computed uniformly effectively
in recursive normal permutations $f'$. \qed
\end{lemma}

If the smallest element of each cycle can be found, the cycles of $f'$ can
be decided as shown in Theorem~\ref{theorem:cycledecchar}. To summarise:

\begin{theorem} \label{theorem:decnormal} Let $f \in \G$. $f$ has decidable
cycles iff $f$ is effectively conjugate to its normal form. A decider $\pi$
for $\Part f$ can be computed from $f'$. \qed
\end{theorem}

The next results show that instead of effective conjugacy to the normal
form, conjugacy to any normal or semi-normal permutation is sufficient.

\begin{lemma} \label{lemma:conjnormal} If a recursive permutation $f$ is
effectively conjugate to a normal permutation $g$ via $h$,
i.e.~$f = h^{-1}gh$, then it is effectively conjugate to its normal form
$f'$ which can be computed from $f$ and $h$.
\end{lemma}

\begin{proof} It follows from the conjugation that the normal permutation
$g = hfh^{-1}$ is recursive, so that $\Part g \cong \Part f = \Part f'$ is
decidable; by Theorem~\ref{theorem:decnormal}, a decider $\gamma$ for
$\Gamma = \Part g$ can be found uniformly effectively from $g$. The proof
of Theorem~\ref{theorem:permconj2} shows that the conjugation $h$ is a
$(\Part f, \Part g)$-isomorphism. By Proposition~\ref{prop:existnormal} we
can compute $f'$ from $f$ using the decider $\lambda xx'[h(x) \Gamma h(x')]$
for $\Part f$ which is uniform in $\gamma$ and $h$.
\end{proof}

The semi-normal form has an incompatibility between the structure of finite
and infinite cycles. This incompatibility provides additional information:
it is possible to decide whether a number lies in a finite or an infinite
cycle, uniformly in the permutation. The next lemma shows that this
information can be discarded effectively to obtain a normal permutation in
the same effective conjugacy class. By Lemma~\ref{lemma:conjnormal}, it
follows that if a permutation $f$ is effectively conjugate to a semi-normal
permutation, it is also effectively conjugate to its normal form and
therefore decidable.

\begin{lemma} \label{lemma:semiisnormal} If $f$ is a recursive semi-normal
permutation, then it is effectively conjugate to its normal form $f'$, which
can be computed from $f$.
\end{lemma}

\begin{proof} We first note that the structure of a semi-normal permutation
allows to decide for every $x$ whether it is in a finite cycle or an
infinite one. Let $x$ be given. Again we inspect the neighbours $x^+ =
f(x)$ and $x^- = f^{-1}(x)$ of $x$. $x$ is the least element in the cycle,
in case of finite as well as infinite $[x]_f$, iff $x \le \min \{x^+,x^-\}$.
Using alternating search, we can again find the smallest element $x_0$ in
the cycle of $x$. By computing $x_0, f(x_0), f^2(x_0)$, we can tell if the
cycle length is $\le 2$. If it is, the cycle is, of course, finite. Else the
values $x_0^-, x_0, x_0^+$ are all distinct. If the cycle is infinite, it is
in normal form and it must be $x_0 < x_0^- < x_0^+$. If the cycle is finite,
it is in semi-normal form and conversely it must be $x_0 < x_0^+ < x_0^-$.
This gives a way to decide cycle finiteness.

To construct the normal form $f'$ of $f$, we determine for an input $x$ if
its cycle is finite or infinite. The infinite cycles are already normal
because $f$ is semi-normal. If the cycle is finite, we can generate a finite
list of all members of the cycle $[x]_f$ by repeated application of $f$ to
$x$. The appropriate image for $x$ to produce a normal cycle can be
determined according to Figure~\ref{fig:normaldiag}.
\end{proof}

\begin{corollary} If $f \in \G$ is effectively conjugate to a semi-normal
permutation, then $f$ has decidable cycles. \qed
\end{corollary}

Semi-normal permutations are studied further in \parref{par:cyclefinite}.
The next subsection deals with permutability criteria for decidable
equivalences. Some sufficient conditions even yield permutability by a
semi-normal permutation, which provides further motivation for
\parref{par:cyclefinite}.

\paragraphtitle{Permutability} \parlabel{par:permutability} Recall that a
decidable equivalence is permutable if its blocks are the orbits of a
recursive permutation. Corollary~\ref{coro:hasuniquenormal} already
formulated a permutability condition which we restate as

\begin{theorem} The recursive normal permutations and the decidable and
permutable equivalences are in bijection, given by $\phi: g \mapsto
\Part g$.
\end{theorem}

\begin{proof} This mapping is well-defined since the equivalence associated
to a recursive normal permutation is decidable by
Theorem~\ref{theorem:decnormal}. To see bijectivity it suffices to find an
inverse mapping. This inverse is the function which associates to every
decidable and permutable equivalence its, by
Corollary~\ref{coro:hasuniquenormal}, unique recursive normal permutation.
\end{proof}

It is sound to represent a decidable permutable equivalence $\Pi$ by any
pair consisting of a decider $\pi$ for $\Pi$ together with an element of
$\Perm \Pi$, as these are witnesses for the asserted properties of the
equivalence. Under this representation, the bijection and its inverse are
computable. In one direction, any recursive normal $f'$ maps to the pair
$(\pi, f')$ which represents $\Part f'$. As Theorem~\ref{theorem:decnormal}
shows, $\pi$ can be obtained from $f'$. In the other direction, $(\pi, f)$
maps to the normal form $f'$ of $f$, which is computable from $\pi$ and $f$
by Proposition~\ref{prop:existnormal}.

A closer inspection of the proof of Proposition~\ref{prop:existnormal} shows
that the algorithm to construct the normal element can be reformulated to
rely on the decidability of $\Pi$ and a way to determine for any $x$ if
there is a greater element in its block. A permutation $f \in \Perm \Pi$
provided such a method in the proof of Proposition~\ref{prop:existnormal}.
We proceed to show the converse: if such a method is available, there must
be an $f \in \Perm \Pi$:

\begin{theorem} \label{theorem:rhochar} Let $\Pi$ be decidable via $\pi$.
$\Pi$ is permutable iff the predicate $\rho(x) = [\exists x' > x: x' \Pi x]$
is recursive. Given $\pi$, such a function $\rho$ can be constructed from
the normal permutation for $\Pi$, and therefore, in the presence of $\pi$,
from any member of $\Perm \Pi$, and vice versa.
\end{theorem}

\begin{proof} ``$\Rightarrow$'': If $\Perm \Pi \not= \emptyset$ there is an
$f' \in \Perm \Pi$ in normal form. Let $x$ be given. By
Lemma~\ref{lemma:normalmin} we can find the smallest number $x_0$ in $P(x)$.
The function $a(j) := f'^{\delta^{-1}(j)}(x_0)$ enumerates $P(x)$ in
increasing order for $0 \le j < |P(x)|$. We can obtain the smallest index
$k = k(x)$ such that $a(k) = x$.

Suppose $P(x)$ is finite of length $n \in \N^+$. Since $a$ is increasing for
its first $n$ arguments, we have $a(n) \le a(n-1)$, because $a(n) \in P(x)$
and $a(n-1)$ is the greatest element therein. By definition $0 \le k(x) < n$.
Now, $x$ is the greatest element in the finite block iff $k(x) = n-1$ iff
$a(k(x)+1) \le ak(x)$, the latter of which can be checked without knowing
$n$. If the block is infinite, there cannot be a greatest element and indeed
$a(k(x)+1) > ak(x)$ will always hold. Therefore $[\exists x' > x: x' \Pi x]
= [ak(x) < a(k(x)+1)]$ gives a uniform way to compute $\rho$.

``$\Leftarrow$'': The construction is the same as in the proof of
Proposition~\ref{prop:existnormal}, except that $\rho$ is used instead of
the function $f$ there to decide $[\exists x' > x: x' \Pi x]$.
\end{proof}

The next goal is to obtain a non-permutable equivalence relation. This can
be achieved by encoding the computation of Turing machines well enough so
that decidability of the criterion in Theorem~\ref{theorem:rhochar} implies
decidability of the Halting Problem. The equivalence is defined over codings
of pairs $\code{x, n}$ where $x$ is a program and $n$ a step counter in the
computation $\varphi_x(x)$. This setting reveals a flaw of
Theorem~\ref{theorem:rhochar}: deciding $\code{x, n} > \code{x', n'}$
requires knowledge of the coding $\lambda xn[\code{x, n}]$. It would suffice
for the current purpose to fix the \emph{standard coding} of pairs
$\code{x, y} := \frac{1}{2}(x^2 + 2xy + y^2 + 3x + y)$ \cite[p.~64]{rogers87},
because it is strictly increasing in its second parameter. Such a fixation
is unpleasant and defining permutations over tuples is a useful tool in
general. We provide at least a variant of Theorem~\ref{theorem:rhochar}
which is independent of the coding of pairs and deals with the kind of
equivalence that is later considered in Proposition~\ref{prop:nonpermut}
and Lemma~\ref{lemma:infcycle}.

\begin{definition} An infinite family $\Pi_z$, $z \in \N$, of decidable
equivalences is \emph{uniformly decidable} if there is a recursive
function $\psi$ such that $\psi(z,x,x') = [x \Pi_z x']$. In this case
the \emph{coproduct equivalence} $\Pi$ defined by
\[
  \code{z,x}\Pi\code{z',x'} :\Leftrightarrow z = z' \wedge x \Pi_z x'
\]
is again decidable.
\end{definition}

The blocks of a coproduct $\Pi$ of a family $\Pi_z$ are of the form
$P(\code{z, x}, \Pi) = \{\code{z, x'} : x' \in P(x, \Pi_z)\}$, i.e.~the
blocks of $\Pi_z$ are prefixed by $z$, to make all blocks across the family
disjoint, and then $\Pi$ is the partition consisting of all these blocks.

\begin{corollary} \label{coro:pairrhochar} Let $\Pi$ be the coproduct of
a uniform family $\Pi_z$ of decidable equivalences. $\Pi$ is permutable iff
the predicate $\rho\code{z,x} = [\exists x' > x : x' \Pi_z x]$ is
recursive.
\end{corollary}

\begin{proof} ``$\Rightarrow$'': Let $f \in \Perm \Pi$. Uniform in $z$, we
obtain functions $f_z = \lambda x[\pi_2 f\code{z,x}] \in \Perm \Pi_z$,
where $\pi_2\code{z,x} := x$ denotes the projection of a pair on its second
component. By Theorem~\ref{theorem:rhochar}, there is a function $\rho_z$
for $\Pi_z$, which can be computed from $f_z$ and a decider $\pi_z$ for
$\Pi_z$, by Proposition~\ref{prop:existnormal} and
Theorem~\ref{theorem:rhochar}. $f_z$ and $\pi_z$, in turn, can be computed
uniformly effectively from $f$, $z$ and a uniform decider $\psi$ for the
family $\Pi_z$. Thus $\lambda\code{z,x}[\rho_z(x)]$ is recursive and has
the desired property.

``$\Leftarrow$'': For any fixed $z$, $\lambda x[\rho\code{z,x}]$ is a $\rho$
function for $\Pi_z$ as in Theorem~\ref{theorem:rhochar}. This gives a
permutation $f_z \in \Perm \Pi_z$ which can be constructed uniformly
effectively from $\rho$, $z$ and $\psi$. Then the function
$f\code{z,x} := \code{z,f_z(x)}$ is a recursive permutation and moreover a
member of $\Perm \Pi$.
\end{proof}

\begin{proposition} \label{prop:nonpermut} There exists a decidable
equivalence which is not permutable.
\end{proposition}

\begin{proof} Define
\begin{align*}
  r_x'(n) &:= [\text{$\varphi_x(x)$ halts after $\le n$ steps}], \\
  r_x(n) &:= \begin{cases}
    \fr t, & n = 0, \\
    r_x'(n-1), & n >  0.
   \end{cases}
\end{align*}
Each $r_x$ is a recursive function which induces a decidable equivalence
$\Pi_x$, as per Proposition~\ref{prop:recdec}. Indeed this family of
equivalences is uniformly decidable because $[r_x(n) = r_x(n')]$ can be
decided uniformly in $x, n, n'$, using a universal Turing machine. Let
$\Pi$ be the coproduct of this family.

\emph{Assume} that $\Pi$ is permutable so that
Corollary~\ref{coro:pairrhochar} yields a recursive function
$\rho\code{x, n} = [\exists n' > n : n' \Pi_x n]$. In particular for
$n = 0$, we can decide $[\exists n' \ge 1 : r_x(n') = \fr t]$ uniformly in
$x$, i.e.~whether $\varphi_x(x)$ halts eventually. This \emph{contradicts}
the undecidability of the Halting Problem.
\end{proof}

\begin{theorem} \label{theorem:rhoperm} Let $\Pi$ be decidable via $\pi$ and
let there be a recursive function $\rho$ such that
\[
  \rho(x) = \begin{cases}
    0, & |P(x)| = \infty, \\
    |P(x)|, & \text{else}. \\
   \end{cases}
\]
Then $\Pi$ is permutable by a semi-normal permutation which can be
constructed from $\pi$ and $\rho$.
\end{theorem}

Without the last addition that there be a semi-normal permutation in $\Perm
\Pi$, a proof would have been immediate from Theorem~\ref{theorem:rhochar}.
Theorem~\ref{theorem:cfdec} in \parref{par:cyclefinite} shows that the
converse of Theorem~\ref{theorem:rhoperm} holds, too. Together with other
results from \parref{par:cyclefinite}, this shows that there are
permutable equivalences which lack a semi-normal element, which is why
Theorem~\ref{theorem:rhoperm} cannot be inferred from
Theorem~\ref{theorem:rhochar} and needs a separate proof.

\begin{proof} Let $x$ be given. With a decider for $\Pi$, a decider for
$P(x)$ can be found uniformly in $x$, by Lemma~\ref{lemma:sigmachar}. Using
$\rho$, it is decidable whether $P(x)$ is finite or infinite. If it is
finite, all elements of $P(x)$ can be found by testing $x' \in P(x)$ until
$\rho(x)$ members are found. These numbers can easily be arranged into a
finite semi-normal cycle. If the cycle is infinite, it is an infinite
decidable set and Lemma~\ref{lemma:buildcycle} plus
Corollary~\ref{coro:buildcyclenormal} provide a method to construct the
infinite semi-normal (i.e.~normal) cycle.
\end{proof}

\begin{corollary} \label{coro:condf} If $\Pi$ is decidable, then each of the
following conditions is sufficient for the permutability of $\Pi$ by a
semi-normal element:
\begin{compactenum}[\hskip 2em(a)]
\item $\Pi$ has only finitely many blocks,
\item $\Pi$ has only blocks of the same cardinality (including $\aleph_0$).
\end{compactenum}
\end{corollary}

\begin{proof} We show that there is a recursive function $\rho$ as in
Theorem~\ref{theorem:rhoperm}.

\begin{inparaenum}[(a)]
\item Let $\Pi = \{P_1, \dots, P_n\}$, $p_i = |P_i|$ and $y_i \in P_i$
arbitrary representatives. Using $y(x) := \mu y[y \in P(x)]$ we can find
the unique index $i$ such that $y(y_i) = y(x)$, then set $\rho(x) := p_i$.
The program for $\rho$ only needs to include the subroutine $y(x)$ and the
finitely many constants $y_i$ and $p_i$, $1 \le i \le n$.

\item The function $\rho$ returning the cardinality of a block is constant
and therefore recursive.
\end{inparaenum}
\end{proof}

It is easy to construct counterexamples to the converses of
Corollary~\ref{coro:condf} and Proposition~\ref{prop:fininfdec}
from~\parref{par:cycledec}. There is a recursive semi-normal permutation,
implying decidable cycles, with infinitely many infinite cycles and one
finite cycle. It is therefore a simultaneous counterexample to the
converses of the corollary and the proposition. Consider the partition $\Pi$
whose blocks consist of all numbers $\ge 2$ which have the same smallest
prime factor, and $0$ and $1$ form another block together. This is a
decidable partition. By Theorem~\ref{theorem:rhoperm} (with $\rho(0) :=
\rho(1) := 2$ and $\rho(x) := 0$ else) we see that $\Pi$ has a semi-normal
$f \in \Perm \Pi$ which is a simultaneous counterexample.

It was shown in the proof of Lemma~\ref{lemma:semiisnormal} that the
cycle structure of a recursive semi-normal permutation makes it possible to
decide the finiteness of the cycle of any given number $x$.
Theorem~\ref{theorem:rhoperm} has a similar direction and also involves
semi-normal permutations. \parref{par:cyclefinite} deals with cycle
finiteness and its relation to semi-normality more systematically.

\paragraphtitle{An order-theoretic characterisation of $\Perm \Pi$}
\parlabel{par:orderperm} This section characterises the elements of $\Perm
\Pi$ for a fixed decidable equivalence $\Pi$ as the maximal elements with
respect to cycle inclusion inside a normal subgroup of $\Aut \Pi$. The merit
of this theorem is that Recursion Theory only appears in the setting; the
characterisation itself makes no use of the language of computability. For
the basic order-theoretic notions needed here, see e.g.~\cite{schroeder16}.

Let $\I\Pi := \{f \in \Aut \Pi : f(P) = P \; \forall P \in \Pi\}$ denote
the set of \emph{block-wise identical} permutations in $\Aut \Pi$. Any
$f \in \Perm \Pi$ achieves $[x]_f = P(x, \Pi)$, which implies $f(P) = P$
for every block $P \in \Pi$. This means $\Perm \Pi \subseteq \I\Pi$ and
furthermore $\I\Pi$ is a normal subgroup of $\Aut \Pi$. For an alternative
view on $\I\Pi$, define a \emph{refinement} of an equivalence $\Pi$ to be an
equivalence $\Pi'$ such that $x \Pi' x' \Rightarrow x \Pi x'$. In this case
we write $\Pi' \le \Pi$. No decidability requirements are attached to this
notion and, in this subsection, the symbol $\Perm \Pi$ shall not imply that
$\Pi$ is decidable. Then $\I\Pi = \bigcup_{\Pi' \le \Pi} \Perm \Pi'$. The
finest and coarsest equivalences, e.g., yield $\I\{\{0\}, \{1\}, \dots\} =
\{\id\}$ and $\I\{\N\} = \G$, but the automorphism group is $\G$ in both
cases.

The relation
\[
  f \lesssim g \,:\Leftrightarrow\, \forall x: [x]_f \subseteq [x]_g
\]
is a preorder on $\I\Pi$. It becomes antisymmetric if the permutations
are collapsed to their cycle equivalences, i.e. if one disregards the
sequence of elements in the cycles of a permutation. By $f \lnsim g$ we
mean $f \lesssim g$ and $g \not\lesssim f$. A permutation $f \in \I\Pi$
is \emph{maximal} if there is no $g \in \I\Pi$ such that $f \lnsim g$,
i.e. $f$ is an upper bound on all elements it is comparable to. The set
of maximal elements of $\I\Pi$ is written $\max \I\Pi$.

\begin{theorem} If $\Pi$ is decidable, then $\Perm \Pi = \max \I\Pi$.
\end{theorem}

\begin{proof} First suppose $\Pi$ is permutable. Let $f \in \I\Pi$ and
$g \in \Perm \Pi$ be arbitrary. From $[x]_f \subseteq P(x)$ it follows that
the restriction $f|_P$ is a permutation of $P$ for every block $P \in \Pi$.
This means that $P$ is a union of cycles of $f$. Since $g \in \Perm \Pi$,
we know $[x]_g = P(x)$ and thus $f \lesssim g$. This shows that every
element of $\Perm \Pi$ is an upper bound on $\I\Pi$ and in particular
maximal. On the other hand, if $g$ is maximal in $\I\Pi$ and $\Pi$ is
permutable, there exists an $f \in \Perm \Pi$ and by the first part of the
proof $g \lesssim f$. Maximality of $g$ then implies $f \lesssim g$, i.e.
$P(x) = [x]_f \subseteq [x]_g \subseteq P(x)$ and it follows equality
everywhere and $g \in \Perm \Pi$.

If $\Pi$ is non-permutable, we wish to show that there are no maximal
elements. Given any $f \in \I\Pi$, we construct a permutation in $\I\Pi$
which is comparable to and strictly greater than $f$. Since
$f \not\in \Perm \Pi$, there is an $z$ such that $[z]_f \subsetneq P(z)$.
Since $P(z)$ is a union of cycles of $f$, $P(z)$ must be the union of at
least two cycles of $f$. Because $\Pi$ is decidable, $P(z)$ is a
recursive set. The equivalence $\{P(z), \overbar{P(z)}\}$ is evidently
decidable and has finitely many blocks, which is a sufficient permutability
condition. By Corollary~\ref{coro:condf} we obtain a permutation $c$, one of
whose cycles is $P(z)$. Then
\[
  f'(x) := \begin{cases}
    f(x), & x \not\in P(z), \\
    c(x), & x \in P(z),
   \end{cases}
\]
replaces the multitude of cycles in $f$, which make up $P(z)$, by a single
cycle. $f'$ is in $\I\Pi$ and strictly greater than $f$.
\end{proof}

The proof shows the slightly stronger statement that every member of
$\Perm \Pi$ is not only a maximum of $\I\Pi$ but also an \emph{upper bound},
i.e.~it is maximal and comparable to every member of~$\I\Pi$. We also
remark that finding the $f' \gnsim f$ in the second part of the proof was
an instance of permutability. We had to find a permutable, not necessarily
decidable equivalence $\Part f'$ such that $\Part f \lneq \Part f' \le \Pi$.
The proof above shows that it is not hard to order individual blocks of an
equivalence into a single cycle of a recursive permutation, at least if the
block is decidable and its size is known (cf.~Corollary~\ref{coro:condf}).
The hard part of permutability is ordering infinitely many blocks into
cycles simultaneously. The proof gives the following picture of the order in
$\I\Pi$ if $\Pi$ is decidable but not permutable: every chain in $\I\Pi$
with a maximal element can be extended by adding a strictly greater element
which assembles one further block of $\Pi$ into a single cycle. From this
perspective at least, chains grow along the blocks of $\Pi$, and $\Pi$ has
infinitely many blocks as it is not permutable.

\begin{corollary} $\Perm = \bigcup_{\text{$\Pi$ dec.}} \max \I\Pi$. \qed
\end{corollary}

\section{Cycle finiteness and unsolvable problems} \label{sec:cfunsolv}

This last section proves negative answers to algorithmic questions
surrounding $\Perm$. The \emph{cycle finiteness problem} is introduced
and it is shown that it is in general unsolvable for permutations with
decidable cycles. The subset $\Permcf$ of $\Perm$ where cycle finiteness
is decidable is characterised by semi-normal permutations. It is shown
that cycle decidability and cycle finiteness problems in $\G$ are
intertwined by one-one reductions and that the maximum one-one degree of
either problem is the Halting Problem.

Furthermore it is shown that conjugacy in $\Perm$ is undecidable, that
$\Perm$ is not enumerable, and that it is not closed under multiplication.
Lastly while $\Perm$ is closed under multiplication with finitary
permutations, it can be shown that there is no constructive proof of this
fact, assuming that a finitary permutation $a$ is encoded as a list of
transpositions, whose product is $a$, and each transposition is encoded as
an ordered pair. A constructive proof can be given for $\Permcf$.

\paragraphtitle{Cycle finiteness} \parlabel{par:cyclefinite} We begin by
constructing a permutation for which it is undecidable if numbers lie in a
finite or an infinite cycle. Such a permutation has already been described
in \cite{lehton09}, by encoding the Halting Problem for Turing machines
into the cycle length. Indeed cycle finiteness is the prototype of Collatz'
original problem, which is the motivation of Lehtonen's paper. His
permutation has the additional property that it can be described using a
case distinction on a decidable partition with 5 blocks, and each case has
the form of an affine-linear function. The construction below will use the
general theory developed so far, which makes it swift but does not yield
similar properties.

\begin{definition} For $f \in \G$, the \emph{cycle finiteness problem} of
$f$ is the decision problem
\[
  \textsc{CF}(f) := \{x : |[x]_f| < \infty\}.
\]
The subset of $\Perm$ consisting of permutations with decidable cycle
finiteness problem is denoted $\Permcf$.
\end{definition}

We use a slight modification of the proof of
Proposition~\ref{prop:nonpermut}. Define the $x$-indexed family of recursive
predicates $r_x'(n)$ by
\[
  r_x'(n) = [\text{$\varphi_x(x)$ halts after $\le n$ steps}].
\]
The family $\Pi_x$ of equivalences which correspond to these recursive
functions via Proposition~\ref{prop:recdec} is uniformly decidable by a
universal Turing machine, which makes their coproduct $\Pi$ a decidable
equivalence. The interpretation of the blocks of $\Pi$ is as follows:
$\code{x,n}$ and $\code{x',n'}$ are in the same block iff they belong to
the same program, $x = x'$, and either both computations (after $n$ and
$n'$ steps) did not halt yet or both halted.

We want to show that $\rho\code{x, n} = [\exists n' > n : r_x'(n) = r_x'(n')]$
is recursive in order to apply Corollary~\ref{coro:pairrhochar}. To given
$\code{x, n}$, simulate the computation $\varphi_x(x)$ for $n+1$ steps. If
it halts after $\le n$ steps, it also halts after $\le n+1$ steps, so
$r_x'(n) = r_x'(n+1)$ and $\rho\code{x,n} = \fr t$. If it halts at the
$(n+1)$-st step, then $r_x'(n') \not= r_x'(n)$ for all $n' > n$ and
$\rho\code{x,n} = \fr f$. The remaining case is that the computation did
not halt after $n+1$ steps, in which case it did not halt after $\le n$
steps either, and $r_x'(n) = r_x'(n+1)$, $\rho\code{x,n} = \fr t$. By
Corollary~\ref{coro:pairrhochar}, there is a member $g \in \Perm \Pi$.

\emph{Assume}, we could decide $|[\code{x,n}]_g| < \infty$, for every pair
$\code{x,n}$. Let a program $x$ be given. Then we could decide whether
$\code{x,0}$ lies in a cycle of finite length, which is the same as
deciding whether the computation of $\varphi_x(x)$ halts eventually.
This \emph{contradicts} the undecidability of the Halting Problem.

\begin{lemma} \label{lemma:infcycle} There is a $g \in \Perm \setminus
\Permcf$. \qed
\end{lemma}

Define the decision problem $\textsc{CF*}$ as follows: given $f \in \Perm$,
a decider $\pi$ for the cycles of $f$ and a number $x$, it is to decide
whether $|[x]_f| < \infty$. A fortiori, this problem is recursively
unsolvable. The \emph{diagonal} problem $\textsc{$\Delta$CF*}$ of
$\textsc{CF*}$ asks, given a program which computes a permutation with
decidable cycles, and a decider for the cycles, if that program itself is in
a finite or an infinite cycle of the permutation. This problem may be
thought of as the $\Perm$ version of the one-parameter Halting Problem
$K := \{x : \text{$\varphi_x(x)$ halts}\}$. Using the Recursion Theorem, one
can reduce the seemingly more general problem $\textsc{CF*}$ to its diagonal
and obtains

\begin{corollary} $\textsc{$\Delta$CF*}$ is recursively unsolvable. \qed
\end{corollary}

We can improve upon the inclusion $\G_1 \subseteq \Perm$ from
Proposition~\ref{prop:fininfdec}, by using essentially the same technique
as in that proof.

\begin{proposition} \label{prop:fininfcf} Every permutation with finitely
many infinite cycles has decidable cycle finiteness problem, i.e.
$\G_1 \subseteq \Permcf$.
\end{proposition}

\begin{proof} Let $f \in \G_1$ and $x_1, \dots, x_n$ be a system of
representatives for the infinite cycles of $f$. By
Proposition~\ref{prop:fininfdec}, $f$ has decidable cycles. Given a number
$x$, we can decide if $x$ belongs to any of the cycles $[x_1]_f, \dots,
[x_n]_f$. This is the case iff $|[x]_f| = \infty$.
\end{proof}

Theorem~\ref{theorem:rhoperm} stated that if a decidable equivalence $\Pi$
possesses a recursive function $\rho$ which returns to each $x$ the size of
$P(x)$, or $0$ if $P(x)$ is infinite, then $\Pi$ is permutable by a
semi-normal element. We see now that not every decidable permutable
equivalence has such a function $\rho$. Take the function $g$ from
Lemma~\ref{lemma:infcycle}: $\Part g$ is decidable and permutable. If such a
function $\rho$ existed for $\Part g$, the computable function
$\lambda x[\rho(x) \not= 0]$ would decide $\textsc{CF}(g)$, which is
impossible. The next theorem links cycle finiteness to the conjugacy classes
of semi-normal permutations:

\begin{theorem} \label{theorem:cfdec} Let $f \in \Perm$. Then the following
statements are equivalent:
\begin{compactenum}[\hskip 2em(a)]
\item $\textsc{CF}(f)$ is decidable,
\item there is a $\rho$ function as in Theorem~\ref{theorem:rhoperm} for
$\Part f$, and
\item $f$ is effectively conjugate to its semi-normal form.
\end{compactenum}
\end{theorem}

\begin{proof} ``(a) $\Rightarrow$ (b)'': Write $\Pi = \Part f$ and $P(x) =
P(x, \Pi)$ as usual. By the assumption we can compute $[|P(x)| < \infty]$.
Let $x$ be given. If $|P(x)| = \infty$, then report $\rho(x) := 0$. Else
$P(x) = [x]_f$ is a finite set which we can enumerate by powers of $f$ on
$x$: determine $n = \mu n[n \ge 1 \wedge f^n(x) = x]$. Then $n$ is the
length of $[x]_f$ and we report correctly $\rho(x) := n$.

``(b) $\Rightarrow$ (c)'': Given $\rho$, Theorem~\ref{theorem:rhoperm}
yields a recursive semi-normal element $f'$ in $\Perm \Pi$, which is the
semi-normal form of $f$. Both permutations are recursive and hence
effectively conjugate by Proposition~\ref{prop:permconj1}.

``(c) $\Rightarrow$ (a)'': Let $f'$ be the semi-normal form of $f$. Because
$f$ and $f'$ are effectively conjugate, $f'$ is recursive. The proof of
Lemma~\ref{lemma:semiisnormal} shows that a recursive semi-normal permutation
can be used to solve its own cycle finiteness problem. Since $\Part f =
\Part f'$, it follows that $\textsc{CF}(f) = \textsc{CF}(f')$ is decidable.
\end{proof}

\begin{corollary} $\Permcf$ is the union of effective conjugacy classes of
recursive semi-normal permutations. \qed
\end{corollary}

\begin{corollary} The recursive semi-normal permutations and the decidable,
permutable equivalences with decidable block finiteness problem are in
bijection via $g \mapsto \Part g$. \qed
\end{corollary}

\paragraphtitle{Conjugacy and enumerability} \parlabel{par:conjenum}
As shown by the Theorems~\ref{theorem:permconj2}, \ref{theorem:decnormal}
and \ref{theorem:cfdec}, conjugacy in $\Perm$ is equivalent to the
solvability of certain decision problems. We proceed to prove that
conjugacy, and therefore the solvability of these problems, cannot be
decided.

To a $g \in \Perm$ define the problem $\textsc{Conj}(g)$, which asks,
given $f \in \Perm$ and a decider for its cycles, to decide whether
$f \sim g$. We describe a permutation $g$ for which this problem is
unsolvable. This immediately implies that conjugacy between two given
members of $\Perm$ can in general not be decided.

\begin{theorem} \label{theorem:undecconj} There is a $g \in \Perm$ such
that $\textsc{Conj}(g)$ is recursively unsolvable.
\end{theorem}

\begin{proof} Define $g$ to be
\[
  g(x) := \begin{cases}
    x, & x \equiv 1 \Mod 2, \\
    0, & x = 2, \\
    x-4, & x \equiv 2 \Mod 4, x \not= 2, \\
    x+4, & x \equiv 0 \Mod 4,
   \end{cases}
\]
i.e.~$g = \cycle{\dots, 6, 2, 0, 4, 8, \dots}$.

The proof is by (truth-table) reduction of $\textsc{CF*}$. Let $f, \pi$ and
$x$ be given. We define another permutation $f'$ in a way that if $[x]_f$
is finite, $f'$ consists only of finite cycles, whereas if $[x]_f$ is
infinite, all cycles in $f'$ are either $1$-cycles or infinite and there
are infinitely many $1$-cycles and one infinite cycle. Therefore $f'$ and
$g$ have the same cycle type iff $[x]_f$ is infinite. Since $f'$ and $g$
have finitely many infinite cycles, cycle type equality is equivalent to
effective conjugacy of $f'$ and $g$, by Corollary~\ref{coro:cycletypeeq}
and Theorem~\ref{theorem:permconj2}. This shows
$|[x]_f| < \infty \Leftrightarrow f' \not\sim g$. It therefore suffices
to construct a permutation $f'$ and a decider $\pi'$ for its cycles,
uniformly in $f, \pi, x$, such that $f'$ has the same cycle type as $g$
iff $[x]_f$ is infinite.

Let $k(x') := \xi k[f^k(x) = x']$. Define the following equivalence relation
$\Pi'$: every $x' \not\in [x]_f$ is alone in a block and so is every
$x' \in [x]_f$ with $k(x') \equiv 1 \mod 2$. The remaining numbers
$x' \in [x]_f$ with $k(x') \equiv 0 \mod 2$ form a block together. This
equivalence is evidently decidable and a decider $\pi'$ can be computed
uniformly in $f$, $\pi$ and $x$. Because $f$ and $\pi$ are available, we can
compute a $\rho$ function for $\Part f$, as in Theorem~\ref{theorem:rhochar}.
This $\rho$ function can be used to define a $\rho$ function $\rho'$ for
$\Pi'$ in the following manner: let $y$ be given. If $y \not\in [x]_f$ or
$y \in [x]_f$ and $k(y) \equiv 1 \mod 2$, then $|P(y, \Pi')| = 1$ and there
is no greater element in the same block. Now assume $y \in [x]_f$ and $k(y)$
even. We have to decide whether there is a $y' > y$ with $y' \in [x]_f$ and
$k(y')$ even. First, using $\rho$, we can determine if there is a $y' > y$
which is also in $[x]_f$. If not, $\rho'(y) := \fr f$. Otherwise we can find
the smallest value $y' > y$ with $y' \in [x]_f$. If $k(y')$ is even, we are
done and report $\rho'(y) := \fr t$. Otherwise we repeat the procedure with
$y'$.

This algorithm lists all elements of $[x]_f$ which are greater than $y$
\emph{in order} until one with even $k$ index is found or the cycle is
exhausted. Thus if the algorithm terminates, we have either witnessed the
existence of a greater element than $y$ in $P(y, \Pi')$ or we have verified
that all elements in $[x]_f \supseteq P(y, \Pi')$ which are greater than
$y$ are not in $P(y, \Pi')$. So if the algorithm terminates, it yields a
correct answer. It remains to prove that it always terminates. If $[x]_f$ is
finite, the algorithm halts at the latest after the cycle is exhausted. If
$[x]_f$ is infinite, then there are infinitely many numbers with even $k$
index, and thus arbitrarily large ones; the algorithm will eventually find
one and terminate. With $\rho'$, Theorem~\ref{theorem:rhochar} gives, still
uniformly in $f$, $\pi$ and $x$, a permutation $f' \in \Perm \Pi'$. If
$[x]_f$ is finite, all cycles in $f'$ are finite. If $[x]_f$ is infinite,
$f'$ consists of infinitely many $1$-cycles and one infinite cycle. This
completes the proof.
\end{proof}

The second task treated in this subsection is enumerability. It may be useful
for various constructions to compute an exhaustive list of all the members
of $\Perm$. This is shown impossible here, i.e.~there is no partial recursive
function $\rho$ such that:
\begin{compactenum}[\hskip 2em(1)]
\item $\forall x \in \dom \rho: \varphi_{\rho(x)} \in \Perm$, and
\item $\forall f \in \Perm \exists x \in \dom \rho: f = \varphi_{\rho(x)}$.
\end{compactenum}

\begin{theorem} \label{theorem:perminenum} $\Perm$ is not recursively
enumerable.
\end{theorem}

There are multiple accessible proofs of this theorem. The first is an
obvious but somewhat technical diagonalisation. The second proof is based
on Corollary~\ref{coro:permnormal} and Kent's result \cite[Thm.~2.1]{kent62}
about the composition series of $\G$. From these two follows that the
subgroup of $\G$ generated by $\Perm$ is already all of $\G$. \emph{If}
$\Perm$ was enumerable, then so would be its group closure, but this
\emph{contradicts} the inenumerability of $\G$, \cite[Ex.~4-6]{rogers87}.
We give another short proof based on a result by van Leeuwen:

\begin{proof} By \cite[Thm.~4]{leeuwen15}, no recursively enumerable set
of partial recursive functions with infinite domains can contain all
involutions. All members of $\Perm$ are permutations and have infinite
domains, but by Lemma~\ref{lemma:oneinfdec}, $\Perm$ contains all recursive
permutations with only finite cycles, in particular all involutions. It
follows that $\Perm$ is not enumerable.
\end{proof}

\paragraphtitle{Difficulty of cycle decidability} \parlabel{par:diffcycledec}
\cite{kent62} provides a tool to obtain permutations with particularly
difficult cycle structure.
Let $W_x := \dom \varphi_x$ denote the standard numbering of recursively
enumerable sets. Recall from \cite[\textsection~7.3]{rogers87} that a set
$P$ is \emph{productive} if there is a partial recursive $\psi$, such that
whenever $W_x \subseteq P$ it follows that $\psi(x)$ is convergent and
$\psi(x) \in P \setminus W_x$. The function $\psi$ is called a
\emph{productive function} for $P$. A set $C$ is \emph{creative} if it is a
recursively enumerable complement of a productive set. One example of a
creative set is the Halting Problem $K := \{x : x \in W_x\}$ whose
complement has the identity as a productive function.

\begin{theorem*}[{\cite[Thm.~1.3]{kent62}}] Let $C$ be a creative set. There
is a recursive permutation $k$ composed of infinitely many infinite cycles,
one of which is $C$.
\end{theorem*}

We remark that the proof of Theorem~1.3 in \cite{kent62} is largely based on
his Lemma~1.4, which does not hold as stated there. It states that if $A$
is a non-empty recursively enumerable set, then there is a recursive
permutation with infinitely many infinite cycles, one of which is the
cylinder $A \times \N \subseteq \N$ \cite[\textsection~7.6]{rogers87}.
However, for $A = \N$, the cylinder $\N \times \N = \N$ and if one of the
cycles of the permutation is $\N$, it cannot have infinitely many cycles.
The proof given in \cite{kent62} shows the assertion under the additional
assumption that $A$ has infinite complement. This result is sufficient to
infer his Theorem~1.3, as creative sets necessarily have infinite
complements.

With this theorem we obtain a recursive permutation $k$ with a creative
cycle. The complement of this cycle is productive. By definition, a
productive set is not recursively enumerable. On the other hand, if $f \in
\Perm$, then every cycle of $f$ is a recursive set, so the complement of
every cycle is recursive, too. It follows

\begin{lemma} \label{lemma:gminusperm} There is a permutation $k \in
\G \setminus \Perm$ with infinitely many infinite cycles, one of which is
the Halting Problem $K$. \qed
\end{lemma}

Alternatively, \cite[Thm.~3.1]{higman90} shows that there is a recursive
permutation with infinitely many infinite cycles and no finite cycles, all
of whose transversals are immune \cite[\textsection~8.2]{rogers87}. Since
an immune set is not recursively enumerable, it follows by
Theorem~\ref{theorem:cycledecchar} that this permutation has undecidable
cycles.

We note that Kent's and Higman's constructions produce a permutation with
infinitely many infinite cycles. Indeed, a permutation with finitely many
infinite cycles is necessarily in $\Perm$ by
Proposition~\ref{prop:fininfdec}.

The fact that there are permutations with undecidable cycles raises the
question of how difficult cycle decidability problems for recursive
permutations can become. The following proposition determines the maximal
one-one degree of cycle decidability problems in $\G$. For the definition of
reducibilities and degrees, the reader is referred to
\cite[\textsection\textsection~6ff.]{rogers87}.

\begin{proposition} \label{prop:cycledegree} For every $f \in \G$, the cycle
decidability problem of $f$ is one-one reducible to the Halting Problem, and
there exists an $f' \in \G$ such that the Halting Problem is one-one
reducible to the cycle decidability of $f'$.
\end{proposition}

\begin{proof} Let $f \in \G$ be arbitrary. With an oracle for the Halting
Problem $K$ we can check uniformly in $x, x'$ whether the function
$\lambda w[\xi k[f^k(x) = x']]$ halts on its own encoding (or any other
number because the function ignores its argument). This is the case iff
$x \equiv_f x'$. The mapping of $\code{x,x'}$ to a program for
$\lambda w[\xi k[f^k(x) = x']]$ can be chosen to be strictly increasing
in the numerical value of $\code{x,x'}$ which makes it a one-one reduction.

For the opposite direction, Lemma~\ref{lemma:gminusperm} shows that there
is a recursive permutation $k \in \G$ of which one cycle is the Halting
Problem. Fix a program $x_0$ such that $\varphi_{x_0}(x_0)$ halts. Then
the Halting Problem can be solved by $x \in K \Leftrightarrow
x \equiv_{k} x_0$ which is a one-one reduction.
\end{proof}

In previous sections we have used conjugacy and normal forms to characterise
the solvability of cycle decidability and cycle finiteness problems. An
interesting fact is that these two classes of problems in $\G$ are
inter-reducible. To prove this, a technical lemma is needed:

\begin{lemma} \label{lemma:oddlength} For every $g \in \G$ there is a
$g' \in \G$ and an embedding $j: \N \to \N$ such that $|[x]_g| < \infty
\Leftrightarrow |[j(x)]_{g'}| < \infty$ and all cycles of $g'$ which
intersect $\rng j$ are infinite or of odd length.
\end{lemma}

\begin{proof} The idea is to create for each $x$ a copy of its cycle, double
its length and then add one further element to it. This preserves cycle
finiteness and makes every finite cycle odd. Define $j(x) := \code{x,x,0}$
and $g'$ by
\begin{gather*}
  \code{x,x,0} \mapsto \code{x,x,1} \mapsto \code{x,x,2} \mapsto \code{x,g(x),0} \\
  \code{x,y,0} \mapsto \code{x,y,1} \mapsto \code{x,g(y),0}, \; y \not= x
\end{gather*}
and fix every triple which does not match the decidable patterns above.
Clearly $g'$ is a permutation and every cycle which contains some
$\code{x,x,0}$ is either infinite or of odd length. Since the cycle
structure of $g$ is transferred into the second component and merely
stretched by the third component in the definition of $g'$, we have
$|[x]_g| < \infty \Leftrightarrow |[j(x)]_{g'}| < \infty$ for all $x$.
$|[j(x)]_{g'}|$ is either infinite or odd.
\end{proof}

\begin{theorem} \label{theorem:interred} In $\G$, the classes of cycle
decidability and cycle finiteness problems are one-one inter-reducible,
in the following sense:
\begin{compactenum}[\hskip .99em(i)]
\item For every $f \in \G$ there is a $g \in \G$ and an embedding $j$ such
  that $x \equiv_f y \Leftrightarrow |[j\code{x, y}]_g| < \infty$.
\item For every $g \in \G$ there is an $f \in \G$ and two embeddings $j, j'$
  such that $|[x]_g| < \infty \Leftrightarrow j(x) \equiv_f j'(x)$.
\end{compactenum}
\end{theorem}

\begin{proof} \begin{inparaenum}[(i)]
\item For every pair $x, y$ define the relation
\[
  i \Pi_{x,y} j \; :\Leftrightarrow \; i = j \, \vee \,
    \forall k \in \Z, |k| \le \max\{i,j\}: f^k(x) \not= y.
\]
$\Pi_{x,y}$ is easily seen to be reflexive, symmetric and transitive, and
thus an equivalence. Indeed $\Pi_{x,y}$ is a family of uniformly decidable
equivalence relations indexed by $\code{x,y}$. Let $\Pi$ denote their
coproduct. The relation $\Pi_{x,y}$ is defined in such a way that for any
$i$ there exists an $i' > i$ with $i' \Pi_{x,y} i$ iff $(i+1)\Pi_{x,y} i$.
By the uniform decidability of the relations we immediately obtain a uniform
$\rho$ function as in Corollary~\ref{coro:pairrhochar} for this family.
Application of this corollary gives $g \in \Perm \Pi$. Observe that the
block of $0$ in $\Pi_{x,y}$ contains infinitely many elements iff
$f^k(x) \not= y \; \forall k \in \Z$. The embedding is
$j\code{x,y} = \code{x,y,0}$.

\item By Lemma~\ref{lemma:oddlength} we find a $g'$ and an embedding $j$
such that $[j(x)]_{g'}$ is finite iff $[x]_g$ is, and $[j(x)]_{g'}$ is
either infinite or of odd length. If $[j(x)]_{g'}$ is infinite, $j(x)$ and
$g'j(x)$ cannot be in the same cycle of $g'^2$. On the other hand, if
$[j(x)]_{g'}$ is finite, its length is odd. Application of $g'$ on such a
cycle $[j(x)]_{g'}$ imposes the structure of a finite cyclic group of odd
order on the cycle with the group operation $g'^kj(x) \cdot g'^lj(x) :=
g'^{k+l}j(x)$. Since $2$ is coprime to the order of this group, $g'^2j(x)$
is a generator and $[j(x)]_{g'^2}$ must contain $g'j(x)$. We have shown that
$|[x]_g| < \infty$ iff $|[j(x)]_{g'}| < \infty$ iff $j(x) \equiv_{g'^2}
g'j(x)$, thus set $f = g'^2$ and $j' = g'j$.
\end{inparaenum}
\end{proof}

\paragraphtitle{Products in $\Perm$} \parlabel{par:products}
Theorem~\ref{theorem:permconj2} shows that effective conjugacy and effective
cycle type equality are equivalent in $\Perm$. Recall that effective cycle
type equality of two permutations $f, g$ means computable isomorphy of their
respective equivalences $\Part f, \Part g$. To formulate this theorem
constructively, there must be constructive representations of these
equivalences, i.e.~they must be decidable and represented by their deciders,
or by an equivalent means, as presented in \parref{par:decequiv}.
In this way $\Perm$ is the maximal domain for this theorem. As discussed in
the introduction, the theorem is a constructive analogue of the well-known
theorem that conjugacy in a symmetric group is equivalent to cycle type
equality, where $\Perm$ plays the role of the full symmetric group. This
raises the question about the algebraic structure of $\Perm$.
$\G \supseteq \Perm$ imposes its group operation on $\Perm$ and it
is evident that \begin{inparaenum}[(a)] \item $\id \in \Perm$ and \item if
$f \in \Perm$ then $f^{-1} \in \Perm$.\end{inparaenum}

This subsection proves that multiplicative closure fails. The proof uses
structural results of \cite{kent62} to infer the existence of $f, g \in
\Perm$ such that $fg \not\in \Perm$ without constructing them. The present
proof could have been given after Corollary~\ref{coro:permnormal} was
established. A second theorem contains a positive result in the direction
of multiplicative closure, namely that $\Perm$ is closed under multiplication
from left and right with \emph{finitary} permutations. However, it is also
shown that there is a fixed permutation $g \in \Perm$ such that cycle
decidability of finitary products cannot be witnessed uniformly. The cycle
finiteness problem introduced in \parref{par:cyclefinite} plays an important
role in the characterisation of the cases where computing the decider is
possible.

By Corollary~\ref{coro:permnormal}, $\Perm$ is a normal subset in $\G$, so
that the subgroup $\gen{\Perm}$ generated by $\Perm$ is a normal subgroup of
$\G$. Clearly $\Perm$ contains all finitary permutations and a non-finitary
one, e.g.~$\delta\succf\delta^{-1} = \cycle{\dots, 5, 3, 1, 0, 2, 4,
6, \dots}$ whose decider is trivial.

\begin{theorem*}[{\cite[Thm.~2.1]{kent62}}] If $N$ is a normal subgroup of
$\G$ that contains a permutation with infinite support, then $N = \G$.
\end{theorem*}

Applying this theorem yields that $\gen{\Perm} = \G$. Since $\Perm$ is
closed under inversion, it follows that every recursive permutation can be
expressed as a finite product of members of $\Perm$. By
Lemma~\ref{lemma:gminusperm} there is a $k \in \G \setminus \Perm$. Then
there is a decomposition $k = f_1\dots f_n$ with $f_i \in \Perm$. Now an
index $1 \le i < n$ must exist such that $f_1\dots f_i \in \Perm$ and
$(f_1\dots f_i)f_{i+1} \not\in \Perm$. This proves

\begin{corollary} \label{coro:permnotgroup} $\Perm$ is not a group. More
precisely, there are $f, g \in \Perm$ such that $fg \not\in \Perm$. \qed
\end{corollary}

The proof did not give examples of such $f, g$. They may be found by
factoring $k \in \G \setminus \Perm$ into members of $\Perm$. Such a
factorisation exists for every such $k$ and always yields counterexamples
to the closure of $\Perm$, as seen above.

While $\Perm$ is not closed under multiplication with itself, one might
expect that changes with finite support do not disturb cycle decidability.
This is only partially true. $\Perm$ is closed under multiplication with
$F$, the set of permutations with finite support, but a decider for such
a product cannot always be computed.

\begin{lemma} \label{lemma:transcf} Let $f \in \Perm$. There exists a
recursive function $\rho\code{x,y}$ such that for all $x, y$ the function
$\varphi_{\rho\code{x,y}}$ decides the permutation $\cycle{x,y}f$ iff
$\textsc{CF}(f)$ is decidable. In this case $\rho$ can be constructed
from $f$, a decider for its cycles and a decider for its cycle
finiteness problem.
\end{lemma}

\begin{figure}
\tikzset{
  ball/.style={
    circle,
    minimum size=0pt,
    inner sep=0pt,
    font=\footnotesize
  }
}
\renewcommand{\thesubfigure}{{\normalfont\alph{subfigure}}}
\makebox[\linewidth][c]{%
\subcaptionbox{$x \equiv_f y$ \label{fig:transmulta}}{%
  \input{transmult_a.tex}%
}
\hskip 1.0cm
\subcaptionbox{$x \not\equiv_f y, \; x \in \textsc{CF}(f)$
\label{fig:transmultb}}{%
  \input{transmult_b.tex}%
}
\hskip 1.0cm
\subcaptionbox{$x \not\equiv_f y, \; x, y \not\in \textsc{CF}(f)$
\label{fig:transmultc}}{%
  \input{transmult_c.tex}%
}
}
\caption{The functional digraphs of $f$ and $\cycle{x,y}f$ under different
assumptions on $x, y$. The \emph{functional digraph} of a recursive function
$f$ is a directed graph on vertices $\N$ with an edge from $x$ to $y$ iff
$f(x) = y$. Cycles of permutations are weakly connected components in the
functional digraph. Barring exchange of $x$ and $y$ the shown cases
(a)---(c) are exhaustive.}
\label{fig:transmult}
\end{figure}
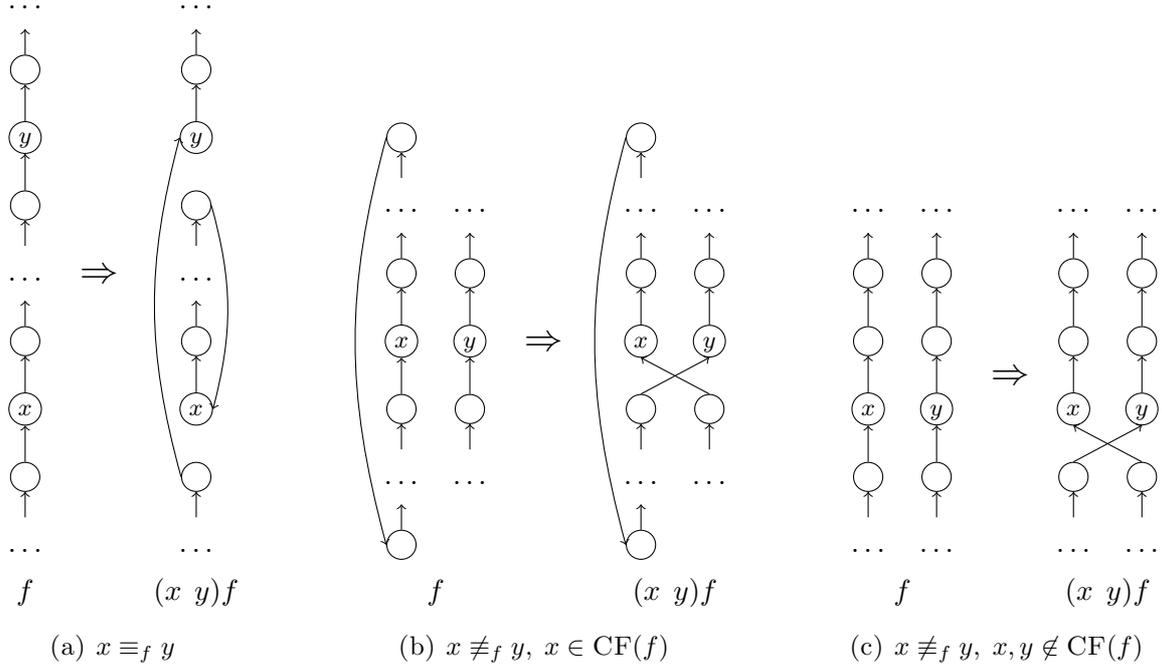

\begin{proof} We begin the proof with ``$\Leftarrow$'' which uses all cases
depicted in Figure~\ref{fig:transmult}. This figure will be used as an
argument in place of a formal calculation in the following proofs.

``$\Leftarrow$'': Suppose $\textsc{CF}(f)$ is decidable. We distinguish
three cases (a)---(c) which correspond to the pictures in
Figure~\ref{fig:transmult}. Which case applies is decidable by our
prerequisites.

\begin{inparaenum}[(a)]
\item If $x \equiv_f y$, then either $x = y$, which is trivial, or there is
an $i > 0$ such that $f^i(x) = y$ or $f^i(y) = x$. By computing $\mu i[i > 0
\wedge \{f^i(x), f^i(y)\} \cap \{x, y\} \not= \emptyset]$, which terminates,
we can find $i$ and determine which of $x$ and $y$ is the ``upper'' and
which is the ``lower'' number in the functional digraph. We can assume, by
symmetry, that $x$ is the lower number, i.e.~$f^i(x) = y$ with $i > 0$. In
$\cycle{x,y}f$ all cycles of $f$ except $[x]_f = [y]_f$ are left untouched.
The cycle $[x]_f$ is split into two cycles: one contains $x, f(x), \dots,
f^{i-1}(x) = f^{-1}(y)$ and the other contains the rest. An algorithm to
decide the cycles of $\cycle{x,y}f$ is obvious.

\item Assume $x \not\equiv_f y$ and $x \in \textsc{CF}(f) \vee y \in
\textsc{CF}(f)$. Without loss of generality let $x \in \textsc{CF}(f)$.
The cycles of $f$ outside of $[x]_f, [y]_f$ are unchanged in $\cycle{x,y}f$.
The cycles $[x]_f$ and $[y]_f$ fuse in $\cycle{x,y}f$ as can be seen in
Figure~\ref{fig:transmultb}: take any two $i, j > 0$, then follow the
arrows $f^{-i}(y), \dots, f^{-1}(y), x, f(x), \dots, f^{-1}(x), y, f(y),
\dots, f^j(y)$.

\item Assume finally $x \not\equiv_f y$ and $x \not\in \textsc{CF}(f) \wedge
y \not\in \textsc{CF}(f)$. Once again the cycles of $f$ apart from $[x]_f$
and $[y]_f$ remain unchanged. In $\cycle{x,y}f$ the cycles $[x]_f$ and
$[y]_f$ are split into four sets and recombined into two different cycles as
Figure~\ref{fig:transmultc} indicates. More specifically: because $[x]_f$ is
infinite, each $x' \in [x]_f$ has a unique number $k_x(x') =
\xi k[f^k(x) = x']$; similarly for $y' \in [y]_f$. Those $x'$ with $k_x(x')
< 0$ are in $[y]_{\cycle{x,y}f}$ and those with $k_x(x') \ge 0$ are in
$[x]_{\cycle{x,y}f}$. Analogously $y' \in [y]_f$ is
$y' \in [y]_{\cycle{x,y}f} \Leftrightarrow k_y(y') \ge 0$ and in
$[x]_{\cycle{x,y}f}$ else.
\end{inparaenum}

Given $f$ and deciders for the cycles of $f$ and its cycle finiteness
problem, the construction of a decider for $\cycle{x,y}f$ is uniform in
$x$ and $y$.

``$\Rightarrow$'': This direction is not uniform in $f$. If $f$
has no infinite cycles, then a decider for $\textsc{CF}(f)$ is constant
and therefore computable. If $f$ has at least one infinite cycle,
let $y_0$ denote a number in an infinite cycle of $f$. We want to decide
$\textsc{CF}(f)$. Let $x$ be given. If $x \equiv_f y_0$, then $x$ is
obviously in an infinite cycle. Else we obtain a decider for
$\cycle{x,y_0}f$. From Figure~\ref{fig:transmult}, with $x \not\equiv_f
y_0$, we see that
\[
  x \in \textsc{CF}(f) \vee y_0 \in \textsc{CF}(f) \Leftrightarrow
  x \equiv_{\cycle{x,y_0}f} y_0
\]
whose RHS is decidable and whose LHS is equivalent to $x \in
\textsc{CF}(f)$, by the choice of $y_0$. This gives a method to decide
cycle finiteness.
\end{proof}

\begin{lemma} \label{lemma:semitranssemi} There is an algorithm uniform in
$f \in \Permcf$, a decider for its cycles and its cycle finiteness problem
which computes to every pair $x, y$ a decider for the cycles and the cycle
finiteness problem of $\cycle{x,y}f$.
\end{lemma}

\begin{proof} Lemma~\ref{lemma:transcf} yields a decider for the cycles of
$\cycle{x,y}f$. Using Figure~\ref{fig:transmult}, a decider for the cycles
of $f$ and for its cycle finiteness problem, it is easy to describe a
decider for the cycle finiteness problem of $\cycle{x,y}f$.
\end{proof}

The transpositions can be enumerated via $\code{x, y} \mapsto \cycle{x,y}$,
for which an algorithm $\tau$ is immediate. The term ``transposition''
is meant to include the improper transposition $\id$, as it is also
enumerated by $\tau$. The finitary permutations $F$ can be enumerated
via the following algorithm which uses the fact that $F$ is the subgroup
generated by transpositions: given any number $z$, decode it into
$z = \code{n, z'}$ and then decode $z' = \code{x_1, x_2, \dots, x_n}$.
This tuple can be mapped uniformly effectively to a program for computing
$\varphi_{\tau(x_1)}\dots\varphi_{\tau(x_n)}$. Call this enumeration~$\eta$.

\begin{theorem} \label{theorem:afbperm} For $f \in \Perm$ and $a, b \in F$
it is $afb \in \Perm$. If $f \in \Permcf$ then $afb \in \Permcf$. Given $f
\in \Permcf$, a decider for its cycles and its cycle finiteness problem,
as well $\eta$ indices for $a$ and $b$, we can find a decider for the
cycles and the cycle finiteness problem of $afb$.
\end{theorem}

\begin{proof} We first show the non-constructive part with $f \in \Perm$.
It suffices to prove the statement under the assumption that $b = \id$:
assume we have shown $\forall f \in \Perm \forall a \in F: af \in \Perm$,
then $afb \in \Perm$ iff $fb \in \Perm$ iff $b^{-1}f^{-1} \in \Perm$ which
is true because $b^{-1} \in F$ and $\Perm$ is closed under inversion. Since
$a$ has finite support, we only need to consider the case where
$a = \cycle{x,y}$ is a transposition and can proceed by induction. Then the
proof is almost taken care of in Lemma~\ref{lemma:transcf}. Indeed reading
the ``$\Leftarrow$'' proof with an oracle for $\textsc{CF}(f)$ (instead of
the assumption that this problem is decidable) gives a decider for
$\cycle{x,y}f$ in each of the three cases considered there. The oracle is
only used to decide which case applied; the constructed deciders are
recursive in each separate case in $f$ and a decider for its cycles only.
Thus a recursive decider always exists and $\cycle{x,y}f \in \Perm$.

The constructive part with $f \in \Permcf$ works in the same way but must
be executed more carefully. Let $a = \varphi_{\eta(z)}$, $b =
\varphi_{\eta(w)}$. The $\eta$ index $z$ encodes a decomposition of $a$ into
transpositions. Using Lemma~\ref{lemma:semitranssemi} inductively on this
decomposition, we obtain deciders for $af$ as follows. At the beginning of
each induction step we have a permutation $g$ (initially $g = f$), deciders
for its cycles and cycle finiteness problem, and the two constituents $x, y$
of the transposition $\cycle{x,y}$. Applying Lemma~\ref{lemma:semitranssemi}
we obtain a decider for the cycles and the cycle finiteness problem of
$\cycle{x,y}g$. The next induction step can then be performed on
$\cycle{x,y}g$.

After this inductive process, we have a decider for the cycles and the cycle
finiteness problem of $af$, which also decide these problems for
$f' = f^{-1}a^{-1}$. From $w$, compute $w'$ such that
$\varphi_{\eta(w')} = b^{-1} \in F$ and apply the same procedure to
$b^{-1}f'$. The deciders we obtain at the end decide the cycles and the
cycle finiteness problem of $b^{-1}f' = b^{-1}f^{-1}a^{-1}$ and also those
of its inverse, $afb$.
\end{proof}

\begin{corollary} \label{coro:permclosure} \
\begin{compactenum}[\hskip .99em(i)]
\item $\Perm$ is closed under multiplication with finitary permutations.
\item $\Permcf$ is constructively closed under multiplication with
finitary permutations.
\item There exists a $g \in \Perm$ such that there is no recursive mapping
of pairs $\code{x, y}$ to a decider for the permutation $\cycle{x,y}g$.
\end{compactenum}
\end{corollary}

The term ``constructively closed'' in the second part of the theorem means
that if $f \in \Permcf$ and $a, b \in F$, witnesses for the statement
$afb \in \Permcf$, i.e. the two deciders, can be computed uniformly in $f$,
$\eta$ indices for $a$ and $b$, and witnesses for $f \in \Permcf$. The
third part of the corollary then means that $\Perm$ is not constructively
closed: $agb$ is in $\Perm$, but a witness for this statement cannot be
computed uniformly in $\eta$ indices for $a$ and $b$, even when $g$ and a
decider for its cycles are fixed.

\begin{proof}
\begin{inparaenum}[(i)]
\item and \item are Theorem~\ref{theorem:afbperm}; \item is
Lemma~\ref{lemma:transcf} applied to the permutation from
Lemma~\ref{lemma:infcycle}.
\end{inparaenum}
\end{proof}

\nocite{*}
\bibliographystyle{alpha}
\bibliography{perm}
\end{document}

%% file: witnesses_a.tex
\begin{tikzpicture}[nodes=draw,font=\small]
  \node[draw=none] at (1,0) {};
  \node[draw=none] at (1,3) {};

  \node[ball,fill=black] (x)    at (1,3)    {};
  \node[draw=none]              at (1,3.4)  {$\strut x$};
  \node[ball,fill=black] (fx)   at (4,3)    {};
  \node[draw=none]              at (4,3.4)  {$\strut f(x)$};
  \node[ball,fill=black] (hx)   at (1,1)    {};
  \node[draw=none]              at (1,0.6)  {$\strut h(x)$};
  \node[ball,fill=black] (ghx)  at (4,1)    {};
  \node[draw=none]              at (4,0.6)  {$\strut gh(x)$};

  \path[every node/.style={font=\sffamily\small}]
    (x.east)   edge[->] node[above] {$\strut f$}               (fx.west)
    (hx.east)  edge[->] node[below] {$\strut g$}               (ghx.west)
    (x.south)  edge[->,dashed] node[midway,left]  {$\strut h$} (hx.north)
    (fx.south) edge[->,dashed] node[midway,right] {$\strut h$} (ghx.north);
\end{tikzpicture}

%% file: witnesses_b.tex
\begin{tikzpicture}
  \node[draw=none] at (1,0) {};
  \node[draw=none] at (1,3) {};

  \node[draw=none] at (1,3) {$\strut f:$};
  \draw[rotate around={-45:(2,2)}]    (2,2)     ellipse (0.9 and 1.6);
  \draw[rotate around={45:(2,2)}]     (2,2)     ellipse (0.4 and 0.7);
  \draw                               (2.3,2.8) ellipse (0.2 and 0.2);
  \draw                               (2.8,2.3) ellipse (0.2 and 0.2);
  \draw[rotate around={45:(1.3,1.3)}] (1.3,1.3) ellipse (0.4 and 0.2);

  \node[draw=none] at (5.5,3) {$\strut g:$};
  \draw[rotate around={-45:(6.5,2)}]   (6.5,2)     ellipse (0.9 and 1.6);
  \draw[rotate around={-30:(6.6,2.3)}] (6.6,2.3)   ellipse (0.4 and 0.7);
  \draw                                (7.448,2.7) ellipse (0.2 and 0.2);
  \draw                                (5.755,1.9) ellipse (0.2 and 0.2);
  \draw[rotate around={-10:(6,1.3)}]   (6,1.2)     ellipse (0.4 and 0.2);

  \draw (2,2)     edge[->,dashed,bend left=15] (6.6,2.3);
  \draw (2.3,2.8) edge[->,dashed,bend left=25] node[pos=0.4,above] {$\strut \theta$} (7.448,2.7);
  \draw (2.8,2.3) edge[->,dashed,bend left=-5] (5.755,1.9);
  \draw (1.3,1.3) edge[->,dashed,bend left=-15] (6,1.2);
\end{tikzpicture}

%% file: normalcycle.tex
\begin{tikzpicture}[->,>=stealth',auto,node distance=2.8em,
  main node/.style={font=\sffamily}]

  \node[main node] (lparen)                   {$\strut ($};
  \node[main node] (anm1)   [right of=lparen, xshift=-1.0em] {$\strut a_{\rho(n-1)}$};
  \node[main node] (dots1)  [right of=anm1]   {$\strut\dots$};
  \node[main node] (a3)     [right of=dots1]  {$\strut a_{\rho(3)}$};
  \node[main node] (a1)     [right of=a3]     {$\strut a_{\rho(1)}$};
  \node[main node] (a0)     [right of=a1]     {$\strut a_{\rho(0)}$};
  \node[main node] (a2)     [right of=a0]     {$\strut a_{\rho(2)}$};
  \node[main node] (dots2)  [right of=a2]     {$\strut\dots$};
  \node[main node] (anm2)   [right of=dots2]  {$\strut a_{\rho(n-2)}$};
  \node[main node] (an)     [right of=anm2]   {$\strut a_{\rho(n)}$};
  \node[main node] (rparen) [right of=an, xshift=-1.6em] {$\strut )$};

  \path[every node/.style={font=\sffamily\small}]
    (a0.north)    edge [bend right] node [left]  {} (a1.north)
    (a1.south)    edge [bend right] node [right] {} (a2.south)
    (a2.north)    edge [bend right] node [left]  {} (a3.north)
    (a3.south)    edge [bend right] node [right] {} (dots2.south)
    (dots2.north) edge [bend right] node [left]  {} (dots1.north)
    (dots1.south) edge [bend right] node [right] {} (anm2.south)
    (anm2.north)  edge [bend right] node [left]  {} (anm1.north)
    (anm1.south)  edge [bend right] node [right] {} (an.south);
\end{tikzpicture}

%% file: transmult_a.tex
\begin{tikzpicture}[nodes=draw,scale=.9] 
  \node[draw=none] (b10)   at (1,1)      {$\strut\dots$};
  \node[ball]      (b11)   at (1,2)      {$\strut$};
  \node[ball]      (b12)   at (1,3)      {$\strut x$};
  \node[ball]      (b13)   at (1,4)      {$\strut$};
  \node[draw=none] (b14)   at (1,5)      {$\strut\dots$};
  \node[ball]      (b15)   at (1,6)      {$\strut$};
  \node[ball]      (b16)   at (1,7)      {$\strut y$};
  \node[ball]      (b17)   at (1,8)      {$\strut$};
  \node[draw=none] (b18)   at (1,9)      {$\strut\dots$};
  \node[draw=none] (l1)    at (1,0.3)    {$f$};

  \node[draw=none,font=\Large] (rarr)   at (2,5)    {$\strut\Rightarrow$};

  \node[draw=none] (b30)   at (3.5,1)      {$\strut\dots$};
  \node[ball]      (b31)   at (3.5,2)      {$\strut$};
  \node[ball]      (b32)   at (3.5,3)      {$\strut x$};
  \node[ball]      (b33)   at (3.5,4)      {$\strut$};
  \node[draw=none] (b34)   at (3.5,5)      {$\strut\dots$};
  \node[ball]      (b35)   at (3.5,6)      {$\strut$};
  \node[ball]      (b36)   at (3.5,7)      {$\strut y$};
  \node[ball]      (b37)   at (3.5,8)      {$\strut$};
  \node[draw=none] (b38)   at (3.5,9)      {$\strut\dots$};
  \node[draw=none] (l2)    at (3.5,0.3)    {$\cycle{x,y}f$};

  \path[every node/.style={font=\sffamily\small}]
    (b10.north) edge[->] (b11.south)
    (b11.north) edge[->] (b12.south)
    (b12.north) edge[->] (b13.south)
    (b13.north) edge[->] (b14.south)
    (b14.north) edge[->] (b15.south)
    (b15.north) edge[->] (b16.south)
    (b16.north) edge[->] (b17.south)
    (b17.north) edge[->] (b18.south)

    (b30.north) edge[->] (b31.south)
      (b31.west) edge[->,bend left=15] (b36.west)
    (b32.north) edge[->] (b33.south)
    (b33.north) edge[->] (b34.south)
    (b34.north) edge[->] (b35.south)
      (b35.east) edge[->,bend left=15] (b32.east)
    (b36.north) edge[->] (b37.south)
    (b37.north) edge[->] (b38.south);
\end{tikzpicture}

%% file: transmult_b.tex
\begin{tikzpicture}[nodes=draw,scale=.9] 
  \node[ball]      (b10)   at (1,1)      {$\strut$};
  \node[draw=none] (b11)   at (1,2)      {$\strut\dots$};
  \node[ball]      (b12)   at (1,3)      {$\strut$};
  \node[ball]      (b13)   at (1,4)      {$\strut x$};
  \node[ball]      (b14)   at (1,5)      {$\strut$};
  \node[draw=none] (b15)   at (1,6)      {$\strut\dots$};
  \node[ball]      (b16)   at (1,7)      {$\strut$};
  \node[draw=none] (b20)   at (2,2)      {$\strut\dots$};
  \node[ball]      (b21)   at (2,3)      {$\strut$};
  \node[ball]      (b22)   at (2,4)      {$\strut y$};
  \node[ball]      (b23)   at (2,5)      {$\strut$};
  \node[draw=none] (b24)   at (2,6)      {$\strut\dots$};
  \node[draw=none] (l1)    at (1.5,0.3)  {$f$};

  \node[draw=none,font=\Large] (rarr)   at (3,4)    {$\strut\Rightarrow$};

  \node[ball]      (b40)   at (4.5,1)      {$\strut$};
  \node[draw=none] (b41)   at (4.5,2)      {$\strut\dots$};
  \node[ball]      (b42)   at (4.5,3)      {$\strut$};
  \node[ball]      (b43)   at (4.5,4)      {$\strut x$};
  \node[ball]      (b44)   at (4.5,5)      {$\strut$};
  \node[draw=none] (b45)   at (4.5,6)      {$\strut\dots$};
  \node[ball]      (b46)   at (4.5,7)      {$\strut$};
  \node[draw=none] (b50)   at (5.5,2)      {$\strut\dots$};
  \node[ball]      (b51)   at (5.5,3)      {$\strut$};
  \node[ball]      (b52)   at (5.5,4)      {$\strut y$};
  \node[ball]      (b53)   at (5.5,5)      {$\strut$};
  \node[draw=none] (b54)   at (5.5,6)      {$\strut\dots$};
  \node[draw=none] (l2)    at (5,0.3)      {$\cycle{x,y}f$};

  \path[every node/.style={font=\sffamily\small}]
    (b10.north) edge[->] (b11.south)
    (b11.north) edge[->] (b12.south)
    (b12.north) edge[->] (b13.south)
    (b13.north) edge[->] (b14.south)
    (b14.north) edge[->] (b15.south)
    (b15.north) edge[->] (b16.south)
    (b16.west) edge[->,bend right=15] (b10.west)
    (b20.north) edge[->] (b21.south)
    (b21.north) edge[->] (b22.south)
    (b22.north) edge[->] (b23.south)
    (b23.north) edge[->] (b24.south)

    (b40.north) edge[->] (b41.south)
    (b41.north) edge[->] (b42.south)
      (b42.north) edge[->] (b52.south)
    (b43.north) edge[->] (b44.south)
    (b44.north) edge[->] (b45.south)
    (b45.north) edge[->] (b46.south)
    (b46.west) edge[->,bend right=15] (b40.west)
    (b50.north) edge[->] (b51.south)
      (b51.north) edge[->] (b43.south)
    (b52.north) edge[->] (b53.south)
    (b53.north) edge[->] (b54.south);
\end{tikzpicture}

%% file: transmult_c.tex
\begin{tikzpicture}[nodes=draw,scale=.9] 
  \node[draw=none] (b10)   at (1,1)      {$\strut\dots$};
  \node[ball]      (b11)   at (1,2)      {$\strut$};
  \node[ball]      (b12)   at (1,3)      {$\strut x$};
  \node[ball]      (b13)   at (1,4)      {$\strut$};
  \node[ball]      (b14)   at (1,5)      {$\strut$};
  \node[draw=none] (b15)   at (1,6)      {$\strut\dots$};
  \node[draw=none] (b20)   at (2,1)      {$\strut\dots$};
  \node[ball]      (b21)   at (2,2)      {$\strut$};
  \node[ball]      (b22)   at (2,3)      {$\strut y$};
  \node[ball]      (b23)   at (2,4)      {$\strut$};
  \node[ball]      (b24)   at (2,5)      {$\strut$};
  \node[draw=none] (b25)   at (2,6)      {$\strut\dots$};
  \node[draw=none] (l1)    at (1.5,0.3)  {$f$};

  \node[draw=none,font=\Large] (rarr)   at (3,3.5)    {$\strut\Rightarrow$};

  \node[draw=none] (b40)    at (4,1)      {$\strut\dots$};
  \node[ball]      (b41)    at (4,2)      {$\strut$};
  \node[ball]      (b42)    at (4,3)      {$\strut x$};
  \node[ball]      (b43)    at (4,4)      {$\strut$};
  \node[ball]      (b44)    at (4,5)      {$\strut$};
  \node[draw=none] (b45)    at (4,6)      {$\strut\dots$};
  \node[draw=none] (b50)    at (5,1)      {$\strut\dots$};
  \node[ball]      (b51)    at (5,2)      {$\strut$};
  \node[ball]      (b52)    at (5,3)      {$\strut y$};
  \node[ball]      (b53)    at (5,4)      {$\strut$};
  \node[ball]      (b54)    at (5,5)      {$\strut$};
  \node[draw=none] (b55)    at (5,6)      {$\strut\dots$};
  \node[draw=none] (l2)     at (4.5,0.3)  {$\cycle{x,y}f$};

  \path[every node/.style={font=\sffamily\small}]
    (b10.north) edge[->] (b11.south)
    (b11.north) edge[->] (b12.south)
    (b12.north) edge[->] (b13.south)
    (b13.north) edge[->] (b14.south)
    (b14.north) edge[->] (b15.south)
    (b20.north) edge[->] (b21.south)
    (b21.north) edge[->] (b22.south)
    (b22.north) edge[->] (b23.south)
    (b23.north) edge[->] (b24.south)
    (b24.north) edge[->] (b25.south)

    (b40.north) edge[->] (b41.south)
      (b41.north) edge[->] (b52.south)
    (b42.north) edge[->] (b43.south)
    (b43.north) edge[->] (b44.south)
    (b44.north) edge[->] (b45.south)
    (b50.north) edge[->] (b51.south)
      (b51.north) edge[->] (b42.south)
    (b52.north) edge[->] (b53.south)
    (b53.north) edge[->] (b54.south)
    (b54.north) edge[->] (b55.south);
\end{tikzpicture}